\documentclass[12pt]{amsproc}
\usepackage{xy, amssymb, hyperref, graphicx}

\hypersetup{pdfauthor={Amritanshu Prasad, Pooja Singla and Steven Spallone}, pdftitle={Conjugacy classes and the extension problem}}
\xyoption{all}

\numberwithin{equation}{section}
\theoremstyle{plain}
\newtheorem{theorem}[equation]{Theorem}
\newtheorem*{theorem*}{Theorem}

\newtheorem*{prop*}{Proposition}
\newtheorem{lemma}[equation]{Lemma}
\newtheorem*{lemma*}{Lemma}
\newtheorem{cor}[equation]{Corollary}
\newtheorem*{cor*}{Corollary}

\newtheorem*{fq*}{Main question}
\newtheorem*{conjecture*}{Conjecture}

\theoremstyle{definition}
\newtheorem{defn}[equation]{Definition}
\newtheorem*{defn*}{Definition}
\newtheorem{remark}[equation]{Remark}
\newtheorem*{remark*}{Remark}
\newtheorem{example}[equation]{Example}

\newtheorem*{notation*}{Notation}

\renewcommand{\arraystretch}{1.2}

\newcommand{\A}{\mathbf A}
\newcommand{\bmat}[4]{{\mat{#1}{\transpose{\vec{#2}}}{\vec{#3}}{#4}}}
\newcommand{\bsl}{\backslash}

\newcommand{\dmatrix}[2]{\begin{matrix}p^{#1_1}& & &\\&p^{#1_2}&&\\&&\ddots&\\&&&p^{#1_{#2}}\end{matrix}}
\newcommand{\nmatrix}[2]{\begin{matrix}-p^{#1_1}& & &\\&-p^{#1_2}&&\\&&\ddots&\\&&&-p^{#1_{#2}}\end{matrix}}

\newcommand{\Ext}{\mathop{E}\nolimits}
\newcommand{\Fq}{\mathbf F_q}
\newcommand{\Fp}{\mathbf F_p}
\newcommand{\gmat}[4]{{\mat{#1}{p\transpose{\vec{#2}}}{\vec{#3}}{#4}}}

\newcommand{\G}{\mathbf G}
\newcommand{\genmat}[3]{\begin{matrix}#1_{11}&#1_{12}&\cdots&#1_{1#3}\\#1_{21}&#1_{22}&\cdots&#1_{2#3}\\\vdots&\vdots&\ddots&\vdots\\#1_{#2 1}&#1_{#2 2}&\cdots&#1_{#2#3}\end{matrix}}
\newcommand{\GL}{\mathop{GL}\nolimits}

\newcommand{\inv}{^{-1}}
\newcommand{\Irr}[2]{\mathrm{Irr}(#1[#2])}
\newcommand{\irr}[1]{\Irr{#1}{t}}

\newcommand{\mat}[4]{
  \begin{pmatrix}
    #1 & #2 \\
    #3 & #4
  \end{pmatrix}
}

\newcommand{\Mat}[9]{
  \begin{pmatrix}
    #1 & #2 & #3 \\
    #4 & #5 & #6 \\
    #7 & #8 & #9
  \end{pmatrix}
}

\newcommand{\ses}[5]{
  \xymatrix{
    0 \ar[r] & #1 \ar[r]^{#4} & #2 \ar[r]^{#5} & #3 \ar[r] & 0
  }
}
\newcommand{\xses}[6]{
  \xymatrix{
    (#1)&0 \ar[r] & #2 \ar[r]^{#5} & #3 \ar[r]^{#6} & #4 \ar[r] & 0
  }
}

\newcommand{\transpose}[1]{#1^{\textsc{t}}}
\newcommand{\Z}{\mathbf Z}

\newcommand{\nc}{\newcommand}
\nc{\ol}{\overline}
\nc{\eps}{\epsilon}
\nc{\pp}{\mathcal{P}}
\nc{\bp}{\mathbf p}
\nc{\bphi}{\mathbf \phi}
\nc{\fL}{\mathcal L}

\nc{\g}{\mathfrak g}
\nc{\dmo}{\DeclareMathOperator}
\dmo{\Gal}{Gal}
\dmo{\pr}{pr}
\dmo{\diag}{diag}

\DeclareMathOperator{\Aut}{Aut}
\DeclareMathOperator{\End}{End}
\DeclareMathOperator{\tr}{tr}

\DeclareMathOperator{\Hom}{Hom}

\title[Similarity classes of matrices]{Similarity of matrices over\\local rings of length two}
\author{Amritanshu Prasad}
\address{AP: The Institute of Mathematical Sciences, CIT campus, Taramani, Chennai 600113, India.}
\email{amri@imsc.res.in}
\author{Pooja Singla}
\address{PS: Department of Mathematics, Indian Institute of Science, Bangalore 560012, India.}
\email{pooja@math.iisc.ernet.in}
\author{Steven Spallone}
\address{SS: Indian Institute of Science Education and Research, Dr. Homi Bhabha Road, Pashan, Pune 411021, India.}
\email{sspallone@iiserpune.ac.in}
\begin{document}
\maketitle

\begin{abstract}
  Let $R$ be a local principal ideal ring of length two, for example, the ring $R=\Z/p^2\Z$ with $p$ prime.
  In this paper we develop a theory of normal forms for similarity classes in the matrix rings $M_n(R)$ by interpreting them in terms of extensions of $R[t]$-modules.
  Using this theory, we describe the similarity classes in $M_n(R)$ for $n\leq 4$, along with their centralizers.
  Among these, we characterize those classes which are similar to their transposes.
  Non-self-transpose classes are shown to exist for all $n>3$.
  When $R$ has finite residue field of order $q$, we enumerate the similarity classes and the cardinalities of their centralizers as polynomials in $q$.
  Surprisingly, the polynomials representing the number of similarity classes in $M_n(R)$ turn out to have non-negative integer coefficients.
\end{abstract}

\section{Introduction}

\subsection{Background}
For any ring $R$, two matrices $A, B \in M_n(R)$ are similar if there exists an invertible matrix $g \in \GL_n(R)$ such that $g A g^{-1} = B$.
The classification of similarity classes in $M_n(R)$ is equivalent to the classification of $R[t]$-module structures on $R^n$ up to isomorphism.
When $R=k$ is a field, $k[t]$ is a principal ideal domain and the structure theorem for finitely generated $k[t]$-modules leads to a classification of similarity classes (see Jacobson \cite[Chapter~3]{basicalgebra1}).

Among the most important cases is that of $R = \Z$ (the ring of rational integers). 
Given matrices $A, B \in M_n(\Z)$, denote their images in $M_n(\Z/m\Z)$ by $A_m$ and $B_m$.
As a possible approach towards solving similarity problem in $M_n(\Z)$, we may ask whether the similarity of $A_m$ and $B_m$ for all $m$ implies the similarity of $A$ and $B$.
Unfortunately this naive approach does not work: for example if
\begin{equation*}
  A = \mat{17\times 11+1}{25\times 11}{11^2}{16\times 11+1} \text{ and } B = \mat{17\times 11+1}{11}{25\times 11^2}{16\times 11+1}
\end{equation*}
then $A_m$ and $B_m$ are similar in $M_2(\Z/m\Z)$ for each positive integer $m$, but $A$ is not similar to $B$ in $M_n(\Z)$ (see Stebe \cite{MR0289666}).

However for $R = \Z_p$, the ring of $p$-adic integers, Applegate and Onishi~\cite{MR656422} proved that $A, B \in M_n(\Z_p)$ are similar if and only if $A_{p^t}, B_{p^t} \in M_n(\Z/p^t\Z)$ are similar for all $t$. 
Therefore solving the similarity problem for $M_n(\Z_p)$ is equivalent to solving the similarity problems for $M_n(\Z/p^t\Z)$ for all $t \geq 1$.

Much effort has been put into the similarity problem for $M_n(\mathbf Z/p^t\mathbf Z)$.  For $R = \Z/p^t \Z$, Davis~\cite{MR0220757} has shown that the problem of similarity of matrices $A, B \in M_n(R)$ is equivalent to the similarity problem of their reductions modulo $p$ provided both $A$ and $B$ satisfy a common polynomial whose reduction modulo $p$ has no repeated roots. In a similar direction, Pomfret~\cite{MR0309963} has shown that for a finite
local ring, invertible matrices whose orders are coprime to the characteristic of the
residue field are similar if and only if their images in the residue field are similar.
Nachaev \cite{MR731899} has classified the similarity classes of $M_3(\Z/p^2 \Z)$, and generalizing his result Avni, Onn, Prasad and Vaserstein \cite{MR2543507} have classified the similarity classes of $M_3(R)$ for $R$ a finite quotient of a complete discrete valuation ring.

According to Nagornyi \cite[Section~4]{MR0498881}), the similarity problem for $M_{4n}(\Z/p^2\Z)$ involves the matrix pair problem for $M_n(\Z/p\Z)$, which is widely believed to be intractable for general $n$ (see Drozd \cite{MR607157}).  Despite this, there have been some interesting recent discoveries for similarity classes in $M_n(\Z/p^2\Z)$, and more generally $M_n(R)$ where $R$ is a local principal ideal ring of length two.  Singla \cite{Poojathesis} has proved that the number and sizes of conjugacy classes of $\GL_n(R)$ depend on $R$ only through the cardinality of the residue field $k$ (for example, the class equations for $\GL_n(\Z/p^2\Z)$ and $\GL_n(\mathbf F_p[t]/t^2)$ are the same).
Jambor and Plesken \cite{Jambor2012250} have proved that the number of similarity classes in $M_n(R)$ that map onto the similarity class of $A \in M_n(k)$ under the natural projection map are in bijective correspondence with the orbits for the action of the centralizer of $A$ in $\GL_n(k)$ on the linear dual of the centralizer of $A$ in $M_n(k)$.
Their results have enabled them to show that the number of conjugacy classes in $\GL_6(\Z/4\Z)$ is $300$ with the help of the GAP computer algebra system.

A fascinating result on similarity classes comes from model theory:
Let $M_i$ denote the number of conjugacy classes in $GL_n(\Z/p^i\Z)$.
Using a result of Denef and van den Dries \cite{denef1988p} on $p$-adic integrals, du Sautoy \cite{MR2105816} has shown that the associated Poincar\'e series 
\begin{equation*}
  Z(t) = \sum_{i=0}^\infty M_i t^i
\end{equation*}
is a rational function of $t$. This is equivalent to saying that the numbers $M_k$ satisfy a finite recurrence relation in $k$.

\subsection{An outline of this article}
In this article, a \emph{partition} is a finite weakly decreasing sequence of nonnegative integers.
Two partitions are considered to be the same if one can be obtained from the other by adding trailing zeroes.
We write $\Lambda$ for the set of all partitions.  Given a partition 
$\lambda=(\lambda_1,\dotsc,\lambda_r) \in \Lambda$, we put $|\lambda|=\sum_{i=1}^r \lambda_i$.  We denote the empty partition by $\emptyset$. Note that $|\emptyset|=0$.

For each partition $\lambda=(\lambda_1,\dotsc,\lambda_r)$ and each irreducible polynomial $p$ with coefficients in $k$, consider the $k[t]$-module
\begin{equation*}
  M_\lambda(p) = k[t]/p^{\lambda_1}\oplus\dotsb\oplus k[t]/p^{\lambda_r}.
\end{equation*}
Denote its automorphism group $\Aut_{k[t]}(M_\lambda(p))$ by $G_\lambda(p)$.

In this article, when $R$ is a local principal ideal ring of length two, we reduce the problem of classifying the similarity classes of matrices in $M_n(R)$ along with their centralizers to the problem of classifying the orbits for the action of the group $G_\lambda(p)$ on
\begin{equation*}
  E_\lambda(p)=\mathrm{Ext}^1_{k[t]}(M_\lambda(p),M_\lambda(p))
\end{equation*}
for all pairs $(\lambda,p)$ where $|\lambda|\deg p\leq n$ (Theorems~\ref{theorem:main} and~\ref{theorem:centralizer}).
Explicit formulas are given for going between similarity classes in $M_n(R)$ and $G_\lambda(p)$-orbits in $E_\lambda(p)$ (Section~\ref{sec:comp-corr}).

In particular, if $R$ and $R'$ are local principal ideal rings of length two with a fixed isomorphism between their residue fields (such as $\Z/p^2\Z$ and $\mathbf F_p[t]/t^2$), then there is a bijection between similarity classes in $M_n(R)$ and $M_n(R')$ which preserves their cardinality (Theorem~\ref{theorem:comparison}).
If their residue fields are perfect, then this bijection is canonical.
\begin{remark}
  \label{remark:ext-vs-hom}
  Jambor and Plesken~\cite{Jambor2012250} work with $\Hom$ rather than with $\mathrm{Ext}$.
  More precisely, they show that, for every matrix $A\in M_n(k)$, the number of similarity classes in $M_n(R)$ whose image contains $A$ is equal to the number of classes for the action of $Z_{GL_n(k)}A$ on $Z_{M_n(k)}A$ given by $g\cdot X = gXg^{-1}$ for $g\in Z_{GL_n(k)}A$ and $X\in Z_{M_n(k)A}A$.
  It follows that the number of similarity classes in $M_n(R)$ is equal to the number of $n$-dimensional representations of of the polynomial algebra $k[X,Y]$, because the isomorphism classes of representations of $k[X, Y]$ where $X$ acts by $A$ is determined by the action of $Y$, which must lie in $Z_{M_n(k)}A$ and is determined up to conjugation by $Z_{GL_n(k)}A$.

  It would be surprising if there were two essentially different approaches to the similarity problem.
  But in fact the functors from pairs of $k[t]$-modules to $k$-vector spaces given by
  \begin{equation*}
    F_1:(M,N)\mapsto \mathrm{Ext}^1_{k[t]}(M,N) \text{ and } F_2:(M,N)\mapsto \Hom_{k[t]}(N,M)'
  \end{equation*}
  (here $V'$ denotes the linear dual of a $k$-vector space $V$) are naturally isomorphic.  Thus their approach is consistent with ours.
\end{remark}

We are able to describe the $G_\lambda(p)$-orbits on $E_\lambda(p)$ for all partitions $\lambda$ with $|\lambda|\leq 4$ (Section~\ref{sec:description-orbits}).
This allows us to classify similarity classes along with their centralizers in $M_n(R)$ for $n\leq 4$.
Among these similarity classes, we characterize the similarity classes which are self-transpose (non-self-transpose classes exist for $n\geq 3$; see Section~\ref{sec:conj-transp}).
When the residue field is finite of cardinality $q$, the numbers of similarity classes and their centralizers are computed as polynomial functions of $q$ by symbolic computer calculations \cite{sagecode}  using the Sage mathematical software \cite{sage}.
This software allows for the possibility of computing the number of similarity classes for $n>4$ once $G_\lambda(p)$-orbits on $E_\lambda(p)$ are enumerated for partitions of larger integers.
In contrast to \cite{Jambor2012250} these calculations are carried out using the order $q$ of the residue field as a symbolic variable, and therefore take care of all possible values of $q$.
The current bottleneck for extending our results beyond $n=4$ is the extension of the results of Section~\ref{sec:description-orbits} to partitions of positive integers $n$ beyond $n=4$.
For $n=5$, the only partitions $\lambda$ that remain to be analyzed are $(3, 2)$ and $(2, 2, 1)$.

Finally, in Section \ref{generalring}, we demonstrate that the basic theory persists for $R$ an arbitrary local ring of length two.  In this generality, one must consider the diagonal action of $G_{\lambda}(p)$ on a product of $E_{\lambda}(p)$'s.

\subsection{Some $q$-rious positivity}
Warnaar and Zudilin \cite{MR2773098} used the phrase \emph{$q$-rious positivity} when they conjectured the positivity of coefficients of naturally occurring polynomials in $q$, in this case, generalizations of Gaussian binomial coefficients.  We have encountered another multitude of such polynomials. 


Just for this subsection, let $R$ be a discrete valuation ring with maximal ideal $P$ and with residue field of order $q$.
For each partition $\lambda=(\lambda_1,\dotsc,\lambda_r)$, consider the $R$-module
\begin{equation*}
  M_\lambda = R/P^{\lambda_1}\oplus\dotsb\oplus R/P^{\lambda_r}.
\end{equation*}
Then $\Aut_R(M_\lambda)$ acts on $\End_R(M_\lambda)$ by conjugation:
\begin{equation}
  \label{eq:similarity-transform}
  g\cdot X = gXg\inv \text{ for }g\in \Aut_R(M_\lambda)\text{ and }X\in \End_R(M_\lambda).
\end{equation}
One may think of the orbits of this action as similarity classes in $\End_R(M_\lambda)$.

Using the results from Avni-Onn-Prasad-Vaserstein \cite{MR2543507} and this paper, the number of $\Aut_R(M_\lambda)$-orbits in $\End_R(M_\lambda)$ are now known for several cases, and these are listed in Table~\ref{table:qurious-table}.
The entries for $(n,n)$ and $(n,n,n)$ correspond to similarity classes in $M_2(R/P^n)$ and $M_3(R/P^n)$, and are taken from \cite{MR2543507}. In the latter entry, $\binom nk_q$ denotes a Gaussian binomial coefficient, which is well-known to be a polynomial in $q$ with non-negative integer coefficients for which combinatorial interpretations exist; see for example Knuth \cite{MR0270933}.
Except for the last one, the remaining entries are inferred from Section~\ref{sec:description-orbits} of this paper. 
\begin{table}
  \caption{Number of $\Aut_R(M_\lambda)$-orbits in $\End_R(M_\lambda)$}
  \label{table:qurious-table}
  \begin{tabular}{cccc}
    \hline
    $\lambda$ & number of orbits & $q=1$  \\
    \hline
    $(n)$, $n>0$ & $q^n$ & $1$ \\
    $(n,1)$, $n>1$ & $q^{n+1}+q^{n}+q$&3 \\
    $(n,n)$, $n>0$ & $q^{2n}+\cdots+q^n$ & $n+1$\\
    $(n,1,1)$ & $q^{n+2}+2q^{n+1}+2q^n+2q^{n-1}$ &7\\
    $(n,n,n)$, $n>0$ &  $q^n\left[\binom{n+2}2_q+\binom n2_q\right]$ & $n^2+n+1$ \\
    $(2,2,2,2)$ & $q^8 + q^7 + 3q^6 + 3q^5 + 5q^4 + 3q^3 + 3q^2$ & 19\\
    $(1^n)$, $n>0$ & $\sum_\lambda q^{\lambda_1}$ & $P(n)$ \\
    \hline
  \end{tabular}
\end{table}
Using the rational normal form, one may show that the number of similarity classes in $M_n(\Fq)$ is a polynomial in $q$ with non-negative integer coefficients whose value at $q=1$ is the number of partitions of $n$:
\begin{equation}
  \label{eq:sim-classes-over-field}
  \text{number of similarity classes in $M_n(\Fq)$} = \sum_\lambda q^{\lambda_1},
\end{equation}
the sum being over all partitions of $n$.
Indeed, we know that every $n\times n$ matrix is similar to a unique block diagonal matrix $C_{p_1}\oplus C_{p_2}\oplus\dotsb \oplus C_{p_k}$, where $p_k|p_{k-1}|\dotsb|p_1$ are monic polynomials in $\Fq[t]$ whose degrees sum to $n$.
Here $C_{p_i}$ denotes the companion matrix of $p_i$.
The polynomials $p_1, p_2, \dotsc, p_k$ are completely determined by their successive quotients $p_1/p_2, p_2/p_3, \dotsc, p_{k-1}/p_k$, and $p_k$, which can be chosen independently.
Fix $\lambda_1, \lambda_2,\dotsc, \lambda_k$ as the degrees of $p_1, p_2,\dotsc, p_k$.
Then there are precisely
\begin{equation*}
  q^{\lambda_1-\lambda_2}\times q^{\lambda_2-\lambda_3}\times \dotsb \times q^{\lambda_{k-1}-\lambda_k}\times q^{\lambda_k} = q^{\lambda_1}
\end{equation*}
choices for $p_1/p_2, p_2/p_2, \dotsc, p_{k-1}/p_k$, and $p_k$, proving (\ref{eq:sim-classes-over-field}).

For matrices with entries in a local principal ideal ring, and more generally, endomorphism algebras of finite torsion modules, there is no theory of rational normal form.
The similarity problem is in fact considered to be wild.  
Therefore, the positivity of polynomials representing the number of similarity classes is a surprising phenomenon which warrants further investigation.

If the observations listed in this section are simply assumed to extend to all partitions $\lambda$, they may be framed as the following conjecture:
\begin{conjecture*}
  For every partition $\lambda$, there exists a monic polynomial polynomial $c_\lambda(q)$ of degree $|\lambda|$ whose coefficients are all non-negative integers such that for every discrete valuation ring $R$ with residue field of order $q$, the number of $\Aut_R(M_\lambda)$-orbits in $\End_R(M_\lambda)$ for the action given by (\ref{eq:similarity-transform}) is $c_\lambda(q)$.
\end{conjecture*}
While the existence of polynomials $c_\lambda(q)$ has been observed before (see, for example, Conjecture~1.3 in \cite{Onn08}), the positivity of coefficients seems to be a new observation.

We end this section by discussing two related situations where positivity is observed:

In 1983, Victor Kac \cite{Kac83} conjectured the positivity of the coefficients of polynomials in $q$ that count the number of isomorphism classes of absolutely indecomposable representations of a quiver over a finite field of order $q$ with a given dimension vector.
From Kac's conjecture it follows that the polynomials that count the total number of isomorphism classes of represenations of a quiver over a field of order $q$ with a given dimension vector also have non-negative integers as coefficients.
Kac's conjecture was proved in complete generality by Hausel, Letellier and Rodriguez-Villegas \cite{HLR13} recently.
It had been proved by Crawley-Boevey and van den Bergh \cite{CBvDB} for indivisible dimension vectors in 2004, and by Mozgovoy \cite{Moz} in 2011 for quivers with at least one loop at each vertex.
By Remark~\ref{remark:ext-vs-hom}, our conjecture for $\lambda = (2^n)$ is nothing but the assertion that the number of isomorphism classes of $n$-dimensional representations of $k[X,Y]$ is a polynomial in $q$ with non-negative integer coefficients.

Anilkumar and Prasad have described an algorithm to compute polynomials in $q$ for the number of $\Aut_R(M_\lambda)$ orbits in $M_\lambda^2$ for the action
\begin{equation*}
  g\cdot(m_1, m_2) = (g\cdot m_1, g\cdot m_2)\text{ for $g\in \Aut_R(M_\lambda)$ and $m_1, m_2\in M_\lambda$}
\end{equation*}
in \cite{AP13}.
For all partitions $\lambda$ of $n$ for $n<20$, these polynomials have been found to have non-negative integer coefficients.

\section{The orbit of extensions associated to a matrix}
\label{sec:orbit-extens-assoc}
Let $R$ be a local principal ideal ring of length two with maximal
ideal $(\pi)$ and residue field $R/(\pi)=k$.

If $k$ has characteristic $0$, then $R$ is necessarily isomorphic to $k[x]/x^2$.  If $k$ is perfect but not of characteristic $0$, then there are two possibilities:  $R=k[x]/x^2$ or $R=W_2(k)$, the quotient of the Witt ring $W(k)$ by the square of its maximal ideal (see Serre \cite[Ch.~II]{MR0354618}).  Note that $W_2(\Fp)=\Z/p^2\Z$. 

An $n\times n$ matrix $\A$ with entries in $R$
determines an $R[t]$-module $M^\A$ as follows: the underlying
$R$-module of $M^\A$ is $R^n$, and $t$ acts as left multiplication by $\A$ when the elements of $M^\A$ are thought of as column vectors.
Taking $\A$ to $M^\A$ gives rise to a bijective correspondence between similarity classes in $M_n(R)$ and isomorphism classes of $R[t]$-modules whose underlying $R$-module is $R^n$.

Let $A$ denote the image of $\A$ in $M_n(k)$ under the residue class map $r:M_n(R)\to M_n(k)$.
Let $M^A$ denote the $k[t]$-module with underlying vector space $k^n$
where $t$ acts by $A$.
This $k[t]$-module can also be viewed as an $R[t]$-module via the
homomorphism $R[t]\to k[t]$ that is obtained by taking the residue classes of coefficients modulo $(\pi)$.

We have a short exact sequence of $R[t]$-modules
\begin{equation}
  \tag{$\xi_\A$}
  \label{eq:basic-short-exact-sequence}
  \ses {M^A}{M^\A}{M^A}\pi p,
\end{equation}
where $\pi$ is used to denote the map $x\mapsto \pi x$ for each $x\in k^n$, and $p:R^n\to k^n$ takes a vector to the residues of its coordinates modulo $(\pi)$.
In this manner, $\A\in M_n(R)$ gives rise to an element $\xi_\A\in \Ext_{R[t]}(M^A,M^A)$ (here $\Ext_{R[t]}$ is abbreviated notation for the functor $\mathrm{Ext}^1_{R[t]}$).

The forgetful functor from the category of $R[t]$-modules to the category of $R$-modules induces a homomorphism
\begin{equation*}
  \omega: E_{R[t]}(M^A,M^A)\to E_R(M^A,M^A).
\end{equation*}
The kernel of $\omega$ consists of those $R[t]$-extensions which split as $R$-extensions; in other words, those extensions
\begin{equation*}
  \ses{M^A}E{M^A}{}{}
\end{equation*}
where $E$ is a vector space over $k$.
But these are the same as the extensions in $\Ext_{k[t]}(M^A,M^A)$.

The result is the exact sequence
\begin{equation}
  \label{eq:sequence-of-exts}
  \xymatrix{
    0\ar[r]& \Ext_{k[t]}(M^A,M^A)\ar[r]^\iota & \Ext_{R[t]}(M^A,M^A) \ar[r]^\omega &\Ext_R(M^A,M^A).
  }
\end{equation}

Let $G_A$ denote the centralizer of $A$ in $\GL_n(k)$.
The group $G_A\times G_A$ acts on each term in
(\ref{eq:sequence-of-exts}) by group automorphisms.
This is because $\mathrm{Ext}$ is a functor in each argument and $G_A$ is contained in the group of automorphisms of $M^A$ in the category of $R[t]$-modules (and therefore also $k[t]$-modules and $R$-modules).
Using the diagonal embedding $G_A\hookrightarrow G_A\times G_A$,
restrict this action to $G_A$ and denote it by $(g,\xi)\mapsto {}^g\xi$.
These actions are discussed in greater detail in Section~\ref{sec:symm-baer-group}.
For now it suffices to note that, for $g\in G_A$ and two extensions $\xi$ and $\xi'$, $\xi'={}^g\xi$ if and only if there exists a morphism $\phi:E\to E'$ such that the diagram 
\begin{equation}
  \label{eq:action}
  \xymatrix{
    (\xi) &  0 \ar[r] & M^A \ar[r] \ar[d]^g & E \ar[r] \ar[d]^\phi & M^A\ar[r]\ar[d]^g & 0\\
    (\xi') & 0 \ar[r] & M^A \ar[r] & E' \ar[r] & M^A\ar[r]& 0
  }
\end{equation}	
commutes.
It follows that the maps $\iota$ and $\omega$ preserve these actions.

Write $\xi_0$ for $\omega(\xi_{\A})$ (which is independent of $\A$).
\begin{lemma}
  \label{lemma:fixed}
  The extension $\xi_0\in E_R(M^A,M^A)$ is invariant under the action
  of $G_A$.
  Consequently, the subset $\omega\inv (\xi_0)$ is preserved by the action of
  $G_A$ on $E_{R[t]}(M^A,M^A)$.
\end{lemma}
\begin{proof}
  By the characterization (\ref{eq:action}) of the $G_A$-action, it
  suffices to show that for all $g \in G_A$, there exists an $R$-module homomorphism
  $\mathbf g:M^{\A}\to M^{\A}$ such that the diagram
  \begin{equation*}
    \xymatrix{
      0 \ar[r] & M^A \ar[r] \ar[d]^g & M^{\A}
      \ar[r] \ar[d]^{\mathbf g} & M^A\ar[r]\ar[d]^g & 0\\
      0 \ar[r] & M^A \ar[r] & M^{\A} \ar[r] & M^A\ar[r]& 0
    }
  \end{equation*}
  commutes.
  But any $\mathbf g\in M_n(R)$ with $r(\mathbf g)=g$ gives such a homomorphism.
\end{proof}
The proof of the following lemma is straightforward:
\begin{lemma}
  \label{lemma:fiber}
  Intersecting a similarity class in $M_n(R)$ with $r\inv(A)$ gives rise to a bijective correspondence
  between the set of similarity classes in $M_n(R)$ whose image in
  $M_n(k)$ is the similarity class of $A$ and the set of $r\inv(G_A)$-orbits
  in $r\inv(A)$.
\end{lemma}
Therefore, the set of $\GL_n(R)$-similarity classes intersecting $r^{-1}(A)$ is the same as the set of $r^{-1}(G_A)$-orbits in $r^{-1}(A)$.

\begin{theorem}
  \label{theorem:fiber}
  The map $\A\mapsto \xi_\A$ induces a bijection from the set of $r\inv(G_A)$-orbits in $r\inv(A)$ to the set of $G_A$-orbits in $\omega\inv(\xi_0)$.
\end{theorem}
\begin{proof}
  Suppose $\A'=\mathbf g\A \mathbf g\inv$ for some $\mathbf g\in r\inv(G_A)$.
  Then the diagram
  \begin{equation}
    \label{eq:same-orbit}
    \xymatrix{
      0 \ar[r] & M^A \ar[r] \ar[d]^g & M^\A \ar[d]^{\mathbf g} \ar[r] & M^A\ar[d]^g\ar[r] & 0\\
      0 \ar[r] & M^A \ar[r] & M^{\A'} \ar[r] & M^A \ar[r] & 0
    }
  \end{equation}
  commutes (here $g$ is the image of $\mathbf g$ in $\GL_n(k)$).
  Therefore, $\xi_{\A'}={}^g\xi_{\A}$,  proving that the map in the statement is well-defined on the level of $r\inv(G_A)$-orbits.
  
  If $\xi_\A$ and $\xi_{\A'}$ lie in the same $G_A$-orbit of $\Ext_{R[t]}(M^A,M^A)$, then there exists $g\in G_A$ and a homomorphism $\mathbf g:M^\A\to M^{\A'}$ of $R[t]$-modules such that the diagram (\ref{eq:same-orbit}) commutes.
  The commutativity of this diagram means that $g=r(\mathbf g)$.
  The fact that $\mathbf g$ is an $R[t]$-module homomorphism means that $\mathbf g \A=\A'\mathbf g$.
  Therefore $\A$ and $\A'$ lie in the same $r\inv(G_A)$-orbit of $r\inv(A)$, showing that the map in the statement is injective.

  For any extension of $R[t]$-modules
  \begin{equation*}
    \xses\xi{M^A}E{M^A}{i}{p}
  \end{equation*}
  which is in $\omega\inv(\xi_0)$, $E$ is isomorphic to $R^n$ as an $R$-module.
  The matrix $\mathbf g$ by which $t$ acts on $R^n$ must lie in $r\inv(A)$ because $p$ is an $R[t]$-module homomorphism.
  The orbit of $\xi$ is the image of the orbit of $\mathbf g$ under the map in the statement, showing that the map is surjective.
\end{proof}
The exact sequence (\ref{eq:sequence-of-exts}) implies that
$\omega\inv(\xi_0)$ is a coset of the image of
$E_{k[t]}(M^A,M^A)$ in $E_{R[t]}(M^A,M^A)$, and therefore is in
bijective correspondence with $E_{k[t]}(M^A,M^A)$.
In fact, a stronger statement is true:
\begin{theorem}
  \label{theorem:coset-translation}
  There exists a bijection between the sets $\omega\inv(\xi_0)$ and $E_{k[t]}(M^A,M^A)$ which preserves the action of $G_A$.
\end{theorem}
\begin{proof}
  For any $\tilde \xi_0\in \omega\inv(\xi_0)$,
  \begin{equation*}
    \iota_{\tilde\xi_0}: \xi\mapsto \iota(\xi)+\tilde\xi_0
  \end{equation*}
  is a bijection between $E_{k[t]}(M^A,M^A)$ and
  $\omega\inv(\xi_0)$.
  Now
  \begin{align*}
    \iota_{\tilde\xi_0}({}^g\xi)&=\iota({}^g\xi)+\tilde\xi_0\\
    &={}^g\iota(\xi)+\tilde\xi_0.
  \end{align*}
  On the other hand,
  \begin{equation*}
    {}^g\iota_{\tilde\xi_0}(\xi)={}^g\iota(\xi)+{}^g\tilde\xi_0.
  \end{equation*}
  Therefore $\iota_{\tilde\xi_0}$ preserves the action of $G_A$ if and only if $\tilde \xi_0$ is fixed by $G_A$.
  The extension $\tilde\xi_0$ is always of the form $\xi_{\A_0}$ (defined at the beginning of this section) for some $\A_0\in M_n(R)$ such that $r(\A_0)=A$.
  The extension $\xi_{\A_0}$ is fixed by $G_A$ if and only if, for
  every $g\in G_A$, there exists $\mathbf
  g\in \Aut_{R[t]}M^{\A_0}$ such that the diagram
  \begin{equation*}
    \xymatrix{
      0 \ar[r] & M^A \ar[r] \ar[d]^g & M^{\A_0} \ar[r] \ar[d]^{\mathbf g} & M^A\ar[r]\ar[d]^g & 0\\
      0 \ar[r] & M^A \ar[r] & M^{\A_0} \ar[r]& M^A\ar[r]& 0
    }
  \end{equation*}
  commutes. In other words, $\xi_{\A_0}$ is fixed by $G_A$ if and only if, for each $g\in G_A$, there exists $\mathbf g$ which commutes with $\A_0$ and such that $r(\mathbf g)=g$.
  This is provided by the following important lemma (which holds for any commutative ring $R$).
\end{proof} 
\begin{lemma}[{\cite [Lemma~5.1.1]{Poojathesis}}]
  \label{lemma:main}
  Let $R$ be a commutative ring, $k$ a field, and $r:R\to k$ a surjective ring homomorphism.
  For every $A\in M_n(k)$, there exists $\A_0\in M_n(R)$ with $r(\A_0)
  = A$ such that, for every $B\in M_n(k)$ which commutes with $A$,  there exists
  $\mathbf B\in M_n(R)$ which commutes with $\A_0$ and satisfies $r(\mathbf B)=B$.
\end{lemma}

We use the method of proof of Jambor and Plesken \cite[Lemma~6]{Jambor2012250}.

\begin{proof}
  Suppose the $k[t]$-module $M^A$ has primary decomposition
  \begin{equation*}
    M^A=\bigoplus_{p} M^A_{p},
  \end{equation*}
  where $p$ runs over a finite set of irreducible polynomials in
  $k[t]$ and $M^A_{p}$ denotes the $p$-primary part of $M^A$,
  which is an invariant subspace for $A$.
  Since $\End_{k[t]}M^A=\bigoplus_{p} \End_{k[t]} M^A_{p}$, it
  suffices to prove the lemma in the case where $M^A$ is
  $p$-primary for a fixed irreducible monic polynomial $p$.
  We may then assume that
  \begin{equation}
    \label{eq:10}
    M^A = k[t]/p^{\lambda_1} \oplus \cdots \oplus k[t]/p^{\lambda_r}
  \end{equation}
  as a $k[t]$-module for some positive integers $\lambda_1\geq
  \dotsb\geq \lambda_r$ and that $A$ is the matrix for multiplication
  by $t$ with respect to the basis
  \begin{equation}
    \label{eq:9}
    e_1, te_1,\dotsc,t^{\lambda_1-1}e_1;e_2, te_2,\dotsc,t^{\lambda_2-1}e_2;\dotsc;e_r,te_r,\dotsc,t^{\lambda_r-1}e_r, 
  \end{equation}
  where $e_1, \dotsc,e_r$ are the elements coming from the unit in the
  $r$ summands of $M^A$ in (\ref{eq:10}).
  Concretely, $A$ is the matrix
  $C_{p^{\lambda_1}}\oplus\dotsb\oplus C_{p^{\lambda_r}}$,
  where, for any monic polynomial $f$, $C_f$ denotes its companion matrix.
  We shall refer to $A$ as a matrix in \hypertarget{link:primary-rational-canonical-form}{\emph{primary rational normal form}}.

  Let $\bp\in R[t]$ be any monic polynomial whose image in $k[t]$ is
  $p$.
  Let $\A_0$ be the matrix for multiplication by $t$ in
  \begin{equation}
    \label{eq:11}
    M^{\A_0}=R[t]/\bp^{\lambda_1}\oplus\cdots \oplus R[t]/\bp^{\lambda_r}
  \end{equation}
  with respect to the basis (\ref{eq:9}), where now $e_i$ is the
  element of $M^{\A_0}$ coming from $1\in R[t]$ in the $i$th summand of (\ref{eq:11}).
  Thus $\A_0=C_{\bp^{\lambda_1}}\oplus\dotsb\oplus
  C_{\bp^{\lambda_r}}$, and therefore $r(\A_0)=A$.

  Matrices which commute with $A$ may be viewed as endomorphisms in $\End_{k[t]}M^A$.
  By the decomposition (\ref{eq:10}), each $\phi\in \End_{k[t]}M^A$ is determined by the values
  \begin{equation*}
    \phi(e_j)=\sum_{i=1}^r c_{ij} e_i,
  \end{equation*}
  for polynomials $c_{ij} \in k[t]$ which satisfy
  $p^{\lambda_i}|c_{ij} p^{\lambda_j}$ for all $i,j$.
  
  If $\lambda_i \leq \lambda_j$, let ${\mathbf c}_{ij}\in R[t]$ be
  any lift of $c_{ij}$.
  If $\lambda_j< \lambda_i$, then $c_{ij}=p^{\lambda_j-\lambda_i}
  d_{ij}$ for some $d_{ij}\in k[t]$.  In this case, pick any lift
  ${\mathbf d}_{ij}\in R[t]$ of $d_{ij}$, and set ${\mathbf
    c}_{ij}=\bp^{\lambda_j-\lambda_i} {\mathbf d}_{ij}$.
  Thus, $\bp^{\lambda_i}|\mathbf c_{ij}\bp^{\lambda_j}$.
  It follows that

  \begin{equation*}
    \tilde \phi(e_j)=\sum_{i=1}^r \mathbf c_{ij}e_i
  \end{equation*}
  defines an endomorphism of the $R[t]$-module $M^{\A_0}$ of
  (\ref{eq:11}); in other words, its matrix $\mathbf B$ commutes with $\A_0$.
  Since $\mathbf c_{ij}$ has image $c_{ij}$ in $k[t]$, it
  follows that the matrix of $\tilde \phi$ is a lift of the matrix of $\phi$.
\end{proof}

We now return to the usual hypotheses on $R$.  The results of this section can be summarized by the following
definition and theorem:

\begin{defn}
  \label{defn:class}
  Given $\xi\in E_{k[t]}(M^A,M^A)$ consider an extension of $R[t]$-modules
  \begin{equation*}
    \ses{M^A}E{M^A}ip
  \end{equation*}
  in the class corresponding to $\iota(\xi)+\xi_{\A_0}\in E_{R[t]}(M^A,M^A)$, where
  $\A_0$ is the lift of $A$ provided by Lemma~\ref{lemma:main}. Then $E$ is a free
  $R$-module of rank $n$.
  Define $\A_\xi$ to be the matrix of multiplication by $t$ in $E$ with
  respect to any $R$-basis which lifts the standard basis of $M^A$
  (thus $\A_\xi$ is defined up to conjugation by matrices in the kernel of
  $r:\GL_n(R)\to \GL_n(k)$).
\end{defn}

\begin{theorem}
  \label{theorem:main}
  The map $\xi\mapsto \A_\xi$ gives rise
  to a bijection from the set of $G_A$-orbits in $E_{k[t]}(M^A,M^A)$
  to the set of similarity classes of matrices in $M_n(R)$ which lie
  above the similarity class of $A$.
  If $k$ is perfect and $A$ is in \hyperlink{link:primary-rational-canonical-form}{primary rational normal form}, this bijection is canonical.
\end{theorem}

\begin{proof}
  This theorem is essentially a consequence of Lemma~\ref{lemma:fiber} and Theorems~\ref{theorem:fiber} and~\ref{theorem:coset-translation}.
  It only remains to explain why the correspondence is canonical when $k$ is perfect and $A$ is in primary rational normal form.
  For this it suffices to show that the construction of $\A_0$ can be done in a canonical manner.
  The only choices involved in the construction of $\A_0$ is the lift $\bp$ of the polynomial $p$.
  In characteristic zero, $R$ is isomorphic to $k[x]/x^2$, and so each coefficient of $p^{\lambda_i}$ can be lifted to the corresponding constant polynomial in $k[x]/x^2$.
  In positive characteristic, we use the fact (see Serre \cite[Chap.~II, Prop.~8(i)]{MR0354618}) that the residue map $R\to k$ admits a canonical section $s:k\to R$, namely, the unique one which is multiplicative.
\end{proof}

\section{Centralizers}
\label{sec:centralizers}
For $\A\in M_n(R)$, let $\G_{\A}$ denote the centralizer in $\GL_n(R)$ of $\A$.
If $\mathbf g\in \GL_n(R)$ commutes with $\A$ then its image $g=r(\mathbf g)\in
\GL_n(k)$ commutes with the image $A=r(\A)$ of $\A$ in $\GL_n(k)$.
Therefore, if we let $\bar G_{\A}=r(\G_{\A})$, then $\bar G_{\A}$ is a
subgroup of $G_A$.
Lemma~\ref{lemma:main} implies that every matrix $A\in M_n(k)$ has a
lift $\A_0$ for which $\bar G_{\A_0}=G_A$.
However, in general, $\bar G_{\A}$ is merely a subgroup of
$G_A$. For example, take $\A=\pi X$ for some $X\in M_n(R)$ whose image in $M_n(k)$ is not a scalar matrix.
It has image $A=r(\A)=0$ in $M_n(k)$, so $G_A=\GL_n(k)$.
However, $\bar G_\A$ is the centralizer of $r(X)$ in $\GL_n(k)$, a proper subgroup.
The general situation is described by the short exact sequence
\begin{equation}
  \label{eq:1}
  \ses{Z_A}{\G_{\A}}{\bar G_{\A}}ir,
\end{equation}
where $Z_A$ denotes the additive group of matrices in $M_n(k)$ that
commute with $A$.  The map $i$ is given by $i(z)=1+\pi {\bf Z}$, for any lift ${\bf Z}$ of $Z$.

\begin{theorem}
  \label{theorem:centralizer}
  For every $\xi\in E_{k[t]}(M^A,M^A)$,
  \begin{equation*}
    \bar G_{\A_\xi}=\mathrm{Stab}_{G_A} \xi.
  \end{equation*}
\end{theorem}
\begin{proof}
  We use the notation of Section~\ref{sec:orbit-extens-assoc}.
  Since $\iota$ preserves $G_A$-actions, an extension $\xi\in E_{k[t]}(M^A,M^A)$ is fixed by $g\in G_A$ if
  and only if $\iota(\xi)$ is.
  Since $\tilde \xi_0$ is fixed by $G_A$, it follows that $\xi_{\A}=\iota(\xi)+\tilde\xi_0$ is
  fixed by $g$ if and only if $\xi$ is.
  This is equivalent to the existence of an $R[t]$-module homomorphism $\mathbf
  g:M^{\A}\to M^{\A}$ such that the diagram
  \begin{equation*}
    \xymatrix{
      0 \ar[r] & M^A \ar[r] \ar[d]^g & M^\A \ar[r] \ar[d]^{\mathbf g} & M^A\ar[r]\ar[d]^g & 0\\
      0 \ar[r] & M^A \ar[r] & M^\A \ar[r]& M^A\ar[r]& 0
    }
  \end{equation*}
  commutes.
  An $R[t]$-module homomorphism $\mathbf g$ in the above diagram is
  an element of $\G_{\A}$ with $r(\mathbf g)=g$.
\end{proof}

\begin{cor} \label{centralizer_computation}
  For any $\xi\in E_{k[t]}(M^A,M^A)$, the centralizer of $\A_\xi$ in $\GL_n(R)$ has cardinality
  $|Z_A||\mathrm{Stab}_{G_A}\xi|$.
  The maximal possible cardinality of the centralizer of a similarity
  class that lies above $A$ is $|Z_A||G_A|$, which is realized by the
  similarity class of $\A_0$ (corresponding to $\xi=0$).
\end{cor}
Corollary~\ref{centralizer_computation} can be used to refine Theorem~\ref{theorem:main}:

\begin{cor}
  \label{cor:centralizer-preserving-corresp-ext-orbits-and-classes}
  When $k$ is a finite field and $A\in M_n(k)$, then for each positive integer $N$, the set of similarity classes in $M_n(R)$ lying above $A$ whose centralizer in $\GL_n(R)$ has $N$ elements is in canonical bijective correspondence with the set
  \begin{equation*}
    \{\xi\in G_A\bsl E_{k[t]}(M^A,M^A)\;:\;|Z_A||\mathrm{Stab}_{G_A}\xi|=N\}.
  \end{equation*}
\end{cor}
It is now possible to compare the similarity classes over two different rings with the same residue field:
\begin{theorem}
  \label{theorem:comparison}
  Let $R$ and $R'$ be local principal ideal rings of length two with maximal ideals $P$ and $P'$ respectively.
  Fix an isomorphism $R/P \cong R'/P'$ of residue fields.
  If $A\in M_n(R/P)$ and $A'\in M_n(R'/P')$ correspond under this isomorphism, then there is a cardinality-preserving bijection between the similarity classes of $M_n(R)$ which lie over the similarity class of $A$ and similarity classes of $M_n(R')$ which lie over the similarity class of $A'$.
  If the residue fields are perfect, then these bijections are canonical.
\end{theorem}
It follows, for example, that there is a canonical cardinality-preserving bijection between the similarity classes in $M_n(\Z/p^2\Z)$ and $M_n(\mathbf F_p[t]/t^2)$.

\section{Primary decomposition}
\label{sec:analysis-of-fibres}
Let $\irr k$ denote the set of monic irreducible polynomials in $k[t]$.

By the theory of elementary divisors (see Jacobson \cite[Section~3.9]{basicalgebra1} and Green \cite{MR0072878}), the similarity classes in $M_n(k)$ (which correspond to isomorphism classes of $n$-dimensional $k[t]$-modules) are in bijective correspondence with functions $c:\irr k\to \Lambda$ such that
\begin{equation}
  \label{eq:20}
  \sum_{p\in \irr k} (\deg p)|c(p)|=n.
\end{equation}
Let $C(n)$ denote the set of all such functions.
By abuse of notation, we shall use $c\in C(n)$ to denote the corresponding similarity class in $M_n(k)$.

For any $p \in \irr k$ and partition $\lambda\in \Lambda$, define a $k[t]$-module
\begin{equation}
  \label{eq:2}
  M_{\lambda}(p)=k[t]/p^{\lambda_1}\oplus\dotsb\oplus k[t]/p^{\lambda_r}.
\end{equation}
In particular, $M_{\emptyset}(p)$ denotes the trivial module.
Let $G_{\lambda}(p)=\Aut_{k[t]}(M_{\lambda}(p))$ and $E_{\lambda}(p)=E_{k[t]}(M_{\lambda}(p),M_{\lambda}(p))$.
The following theorem is a consequence of properties of the primary decomposition for $k[t]$-modules:
\begin{theorem}
  \label{theorem:types}
  Let $A\in M_n(k)$ be a matrix in the similarity class $c\in C(n)$.
  Then
  \begin{equation*}
    G_A\bsl E_{k[t]}(M^A,M^A)=\prod_{p\in \irr k} G_{c(p)}(p)\bsl E_{c(p)}(p).
  \end{equation*}
\end{theorem}
This is of course a finite product, since only those $p$ for
which $c(p)\neq \emptyset$ contribute.
\begin{remark}
  \label{remark:primary}
  It follows from the above discussion that any $\A\in M_n(R)$ can be written as a direct sum
  \begin{equation*}
    \A = \bigoplus_{p\in \irr k}\A_p,
  \end{equation*}
  where $\A_p$ has $p$-primary image in $M_n(k)$.
\end{remark}
\section{Matrix theory of extensions}
\label{sec:extensions-as-matrices}

In order to obtain the similarity classes in $M_n(R)$ using
Theorem~\ref{theorem:main} it is necessary to find representatives for the orbits of
the action of $G_{\lambda}(p)$ on $E_\lambda(p)$, as explained in the preceding section.
This is achieved by representing elements of
$E_\lambda(p)$ by matrices and reducing them
to normal forms.
This idea  was introduced in \cite[Section~10]{Prasad2010} in a
special case.

Recall that a partition is a finite non-increasing sequence of non-negative integers.
Throughout, the notation
$\lambda=(\lambda_1,\dotsc,\lambda_r)$, $\mu=(\mu_1,\dotsc,\mu_m)$ and
$\nu=(\nu_1,\dotsc,\nu_n)$ will be used.
The notation from (\ref{eq:2}) for $k[t]$-modules will be abbreviated
in this section; the polynomial $p \in \irr k$ will be fixed,
and $M_{\lambda}(p)$ will be denoted simply by $M_\lambda$.

\subsection{Extensions represented as matrices}
\label{sec:bases-extensions}
Let
\begin{equation}
  \label{eq:extension}
  \xses\xi{M_\mu}M{M_\nu}\iota q
\end{equation}
be an extension of $k[t]$-modules.
Denote by $(e_1,\dotsc,e_m)$ and $(f_1,\dotsc,f_n)$ the bases of coordinate vectors in $M_\mu$ and $M_\nu$ respectively.
For each $1\leq i\leq n$, let $\tilde f_i$ be a pre-image of $f_i$ in $M$, meaning that $q(\tilde f_i)=f_i$.
The set
\begin{equation}
  \label{eq:generators}
  \iota(e_1),\dotsc,\iota(e_m),\tilde f_1,\dotsc,\tilde f_n
\end{equation}
generates $M$ as a $k[t]$-module.
Clearly, $p^{\mu_i}\iota(e_i)=0$ for each $1\leq i\leq m$.
Also, since $q(p^{\nu_j}\tilde f_j)=p^{\nu_j}f_j=0$, we have $p^{\nu_j}\tilde f_j\in \iota(M_\nu)$.
Therefore, there exists an $m\times n$ matrix $E(\xi)=(\epsilon_{ij})$ of polynomials such that
\begin{equation}
  \label{eq:relation}
  p^{\nu_j}\tilde f_j =\sum_{i=1}^m{\epsilon_{ij}}\iota(e_i) \text{ for every }1\leq j\leq n.
\end{equation}
Obviously, the value of $\epsilon_{ij}$ matters only modulo $p^{\mu_i}$.
Any other lift of $f_j$ is of the form $\tilde f_j+\iota(a)$ for some $a\in M_\nu$.
Therefore, the value of $\epsilon_{ij}$ is determined modulo $p^{\nu_j}$ by the extension $(\xi)$.
Thus $\epsilon_{ij}$ is well-defined modulo $p^{\min(\mu_i,\nu_j)}$ for all $(i,j)$.

In matrix form, a complete set of relations between the generators (\ref{eq:generators}) of $M$ is given by the matrix equation:

\scalebox{0.8}{
  \begin{minipage}{\textwidth}
    \begin{equation*}
      \label{eq:relations}
      \begin{pmatrix}\iota(e_1)&\cdots&\iota(e_m)&\tilde f_1&\cdots&\tilde f_n\end{pmatrix}
      \mat{\dmatrix\mu m}{\genmat\epsilon mn}{}{\nmatrix \nu n}=\vec 0.
    \end{equation*}
  \end{minipage}
}
\subsection{Baer Equivalence}
\label{sec:Baer-equivalence}
Two extensions $\xi$ and $\xi'$ are said to be \emph{Baer equivalent} if there exists a commutative diagram
\begin{equation}
  \label{eq:Baer-equivalence}
  \xymatrix{
    (\xi) & 0 \ar[r]& M_\mu\ar[r]^\iota\ar@{=}[d] & M\ar[r]^q\ar[d]^\phi & M_\nu\ar[r]\ar@{=}[d] & 0\\
    (\xi') & 0 \ar[r] & M_\mu \ar[r]^{\iota'} & M' \ar[r]^{q'} & M_\nu \ar[r] & 0.
  }
\end{equation}
The set of Baer equivalence classes of extensions over $M_\nu$ with
kernel $M_\mu$ (which turns out to be a $k[t]$-module) is called the
Baer group and coincides with the Ext functor $E_{k[t]}(M_\nu,M_\mu)$ \cite[Chap.~XIV]{homologicalalgebra}. 

Let $N_{\nu, \mu}$ be the $k[t]$-submodule of  matrices $(n_{ij}) \in M_{m \times n}(k[t])$ with $p^{\min( \mu_i,\nu_j )} \mid n_{ij}$ for all $i,j$.
As explained in Section~\ref{sec:bases-extensions}, the extension $(\xi)$ defines an element 
\begin{equation}
  E(\xi)= (\epsilon_{ij}) \in M_{m \times n}(k[t])/N_{\mu,\nu}.
\end{equation}

\begin{theorem}
  \label{theorem:extension-matrix}
  The map which takes the extension $\xi$ to the matrix $E(\xi)$
  is a $k[t]$-module isomorphism from $E_{k[t]}(M_\nu,M_\mu)$ to $M_{m\times n}(k[t])/N_{\mu,\nu}$.
\end{theorem}
\begin{proof}
  For injectivity, note that if $\xi$ and $\xi'$ are isomorphic as in (\ref{eq:Baer-equivalence}), we may choose the lifts of $f_i$ in $M$ and $M'$ so as to correspond under $\phi$.
  It will then follow, from the commutativity of the diagram (\ref{eq:Baer-equivalence}), that the associated matrices $E(\xi)$ and $E(\xi')$ are the same.
  
  For surjectivity, given a polynomial matrix $E=(\epsilon_{ij})$, we can always construct a group $A$ with formal generators
  \begin{equation*}
    \iota(e_1),\dotsc,\iota(e_n), \tilde f_1,\dotsc,\tilde f_m
  \end{equation*}
  and relations given by (\ref{eq:relations}), define $\iota:M_\nu\to M$ by $e_j\mapsto \iota(e_j)$, and $q:A\to M_\mu$ by $\tilde f_i\mapsto f_i$ and $\iota(e_j)\mapsto 0$ for all $(i,j)$.
  The resulting extension $\xi$ will have $E(\xi)=E$. 
\end{proof}

\subsection{Symmetries of the Baer group}
\label{sec:symm-baer-group}
Recall that, for every partition $\lambda$, $G_\lambda$ denotes the automorphism group $\Aut_{k[t]}(M_\lambda)$ of $M_\lambda$.
Given $\xi\in E_{k[t]}(M_\nu,M_\mu)$ as in (\ref{eq:extension}), and
$g\in G_\mu$, let $g\cdot \xi$ be the extension
\begin{equation*}
  \xymatrix{
    (g\cdot \xi) & 0 \ar[r] & M_\mu \ar[r]^{\iota\circ g^{-1}} & M \ar[r]^{q'} & M_\nu \ar[r] & 0
  }
\end{equation*}
We have $g'\cdot(g\cdot\xi)=(g'g)\cdot\xi$, or in other words, that $(g,\xi)\mapsto g\cdot\xi$ is a left action of $G_\mu$ on $E_{k[t]}(M_\nu,M_\mu)$.
Similarly, for $g\in G_\nu$, define $\xi\cdot g$ by
\begin{equation*}
  \xymatrix{
    (\xi\cdot g) & 0 \ar[r] & M_\mu \ar[r]^\iota
    & M \ar[r]^{g\inv \circ q} & M_\nu \ar[r] & 0
  }.
\end{equation*}
The map $(\xi,g)\mapsto \xi\cdot g$ is a right action of $G_\nu$ on $E_{k[t]}(M_\nu,M_\mu)$, with commutes with the action of $G_\mu$ that we described above.
When $\lambda=\mu$, then the extension $\xi'={}^g\xi$ in (\ref{eq:action}) is $g\cdot \xi\cdot g\inv$.
\subsection{Matrix transformations}
\label{sec:matr-transf}
Let us now specialize to $\mu=\nu=\lambda$.
Let $p^\lambda k[t]^r$ denote the submodule of $k[t]^r$ generated by $\{p^{\lambda_i}e_i  \mid 1\leq i\leq m\}$.
Thus $M_\lambda=k[t]^r/p^\lambda k[t]^r$.
Every endomorphism $g$ of $M_\lambda$ lifts to an endomorphism of
$k[t]^r$ which preserves $p^\lambda k[t]^r$.
Thus $g$ is represented by a matrix $(g_{ij})\in  M_r(k[t])$.  We have
\begin{equation}
  \label{eq:matrix}
  g(e_j)=\sum_{i=1}^m g_{ij}e_i.
\end{equation}

Let $L_\lambda$ be the subalgebra of $M_r(k[t])$ which preserves
$p^\lambda k[t]^r$. Concretely,
\begin{equation*}
  L_{\lambda}=  \{x \in M_r(k[t]) \mid p^{\max(0,\lambda_i-\lambda_j) } \text{ divides } x_{ij} \text{ for all } i,j \}
\end{equation*}
Let $I_\lambda$ denote the two-sided ideal in $L_\lambda$ consisting
of endomorphisms which map $k[t]^r$ into $p^\lambda k[t]^r$:
\begin{equation*}
  I_{\lambda}=\{ x \in M_r(k[t]) \mid p^{\lambda_i} \text{ divides } x_{ij} \text{ for all } i,j \}.
\end{equation*}
In this way we have a ring isomorphism
\begin{equation}
  \End_{k[t]}(M_{\lambda}) \cong L_{\lambda}/I_{\lambda}.
\end{equation}
Under this isomorphism we have
\begin{equation}
  G_{\lambda} \cong \left(L_{\lambda} /I_{\lambda} \right)^*.
\end{equation}
Thus automorphisms in $G_\lambda$ are represented by
matrices $(g_{ij})$ with entries in $k[t]$ which are subject to the condition
that $g_{ij}$ is divisible by
$p^{\max(0,\lambda_i-\lambda_j)}$. Two matrices $g$ and
$g'$ represent the same element of $G_\lambda$ if and only if
$g_{ij}\equiv g'_{ij} \mod p^{\lambda_i}$.
Composition of automorphisms corresponds as usual to matrix multiplication. 

Recall that an extension $\xi\in E_{k[t]}(M_\nu,M_\mu)$ is represented
by an $m\times n$ matrix $E(\xi)$ with entries in $k[t]$ (Theorem~\ref{theorem:extension-matrix}).
For $g\in L_\nu$, let $g^\#$ denote the matrix defined by
\begin{equation*}
  g^\#_{ij}=p^{\nu_j-\nu_i}g_{ij},
\end{equation*}
which also has entries in $k[t]$.
\begin{theorem}
  \label{theorem:tranformation-of-extensions}
  For any $\xi\in E_{k[t]}(M_\nu,M_\mu)$, $g\in G_\mu$, and $g'\in G_\nu$, we have
  \begin{equation*}
    E(g\cdot \xi\cdot g')=gE(\xi)(g')^\#,
  \end{equation*}
  where on the right hand side, $g$ and $g'$ are viewed as
  matrices, the product is a matrix product, and the matrix
  $(g')^\#$ is as above.
\end{theorem}
\begin{proof}
  Using the fact that $\iota(e_i)=\iota\circ
  g\inv(g(e_i))$ and (\ref{eq:matrix}), (\ref{eq:relation}) gives
  \begin{equation*}
    p^{\nu_j}\tilde
    f_j=\sum_{i=1}^m\epsilon_{ij}(\iota\circ g\inv)\Big(\sum_{k=1}^m g_{ki}e_k\Big).
  \end{equation*}
  The above equation is a defining equation analogous to
  (\ref{eq:relation}) for $E(g\cdot \xi)$ and gives its $(k,j)$th entry to be
  $\sum_i g_{ki}e_{ij}$, showing that $E(g\cdot\xi)=gE(\xi)$.

  Similarly, suppose that the matrix representation of $g'$ (analogous to
  (\ref{eq:matrix})) is given by
  \begin{equation*}
    g'(f_j)=\sum_{k=1}^n g'_{kj}f_k,
  \end{equation*}
  with $g'_{kj}\in k[t]$.
  A lift of $f_j$ under $g\inv\circ q$ can be obtained by taking
  a lift under $q$ of $g(f_j)$, namely
  \begin{equation*}
    \sum_{k=1}^n g'_{kj}\tilde f_k.
  \end{equation*}
  By (\ref{eq:relation}),
  \begin{align*}
    p^{\nu_j} \sum_{k=1}^n g'_{kj}\tilde f_k &= \sum_{k=1}^n
    g'_{kj} p^{\nu_j-\nu_k}p^{\nu_k}\tilde f_k\\
    &=\sum_{k=1}^n g'_{kj} p^{\nu_j-\nu_k} \sum_{i=1}^m \epsilon_{ik}\iota(e_i)\\
    &=\sum_{i=1}^m\sum_{k=1}^n \epsilon_{ik} p^{\nu_j-\nu_k}g'_{kj} \iota(e_i)\\
    &=\sum_{i=1}^m\sum_{k=1}^n \epsilon_{ik} (g')^\#_{kj} \iota(e_i)\\
  \end{align*}
  which gives the defining equation analogous to
  (\ref{eq:relation}) for the extension $\xi\cdot g'$, showing
  that $E(\xi\cdot g')=E(\xi)(g')^\#$.
\end{proof}

\subsection{The Birkhoff moves}
\label{sec:birkhoff-moves}
We now return to the problem (considered at the beginning of this
section) of finding representatives for the $G_\lambda$-orbits in
$E_\lambda$.
Theorem~\ref{theorem:tranformation-of-extensions} translates
the problem into one of matrix reduction.
The permitted row and column operations arise from
a set of generators of $G_\lambda$.
Our goal is to use these operations to reduce matrices to normal forms.

Following Birkhoff~\cite{GarrettBirkhoff01011935}, we may choose as
generators of $G_\lambda$ the transformations on $(x_1,\dotsc,x_l)\in M_\lambda$ given by
\begin{gather}
  \label{eq:3}
  \tag{$B1$}
  x_i\mapsto x_i+p^{\lambda_i-\lambda_j}\alpha x_j \text{ for some } \alpha\in
  k[t] \text{ and } {i<j}\\
  \label{eq:4}
  \tag{$B2$}
  x_i\mapsto x_i+\alpha x_j \text{ for some } \alpha\in k[t] \text{ and } {j<i}\\
  \label{eq:5}
  \tag{$B3$}
  x_i\mapsto \alpha x_i \text{ for some } \alpha\in k[t] \text{ with } (p,\alpha)=1\\
  \label{eq:6}
  \tag{$B4$}
  x_i \leftrightarrow x_j \text{ when } \lambda_i=\lambda_j \text{
    (interchange of coordinates) }
\end{gather}
The resulting transformations on a matrix in $E_\lambda$ are:
\begin{gather}
  \tag{$E1$}\label{eq:12}
  R_i\mapsto R_i+p^{\lambda_i-\lambda_j} \alpha R_j;\quad C_j\mapsto
  C_j-\alpha C_i \text{ for } i<j,\\
  \tag{$E2$}\label{eq:13}
  R_i\mapsto R_i+\alpha R_j;\quad C_j\mapsto C_j-p^{\lambda_j-\lambda_i}\alpha C_j
  \text{ for } j<i,\\
  \tag{$E3$}\label{eq:14}
  R_i\mapsto \alpha R_i;\quad C_i\mapsto \alpha^{-1}C_i,\\
  \tag{$E4$}\label{eq:15}
  R_i\leftrightarrow R_j;\quad C_i\leftrightarrow C_k \text{ when } \lambda_i=\lambda_j,
\end{gather}
with $\alpha\in k[t]$ for (\ref{eq:3}) and (\ref{eq:4}) and $\alpha\in k[t]$
with $(p,\alpha)=1$ in (\ref{eq:5}). In (\ref{eq:5}),
$\alpha\inv$ denotes the inverse of $\alpha$ modulo a sufficiently large power
of $p$ (larger than $\lambda_1$).  Here $R_i$ and $C_i$ refer to the $i$th row and column, respectively, of the given matrix.
We summarize the preceding discussion as a theorem:
\begin{theorem}
  \label{theorem:Birkhoff-moves}
  Two matrices $E$ and $E'$ in $E_\lambda$ lie in the same
  $G_\lambda$-orbit if and only if there is a sequence of
  operations of the form (\ref{eq:12}),...,(\ref{eq:15}) which transform $E$ to $E'$.
\end{theorem}

\section{Computation of the correspondence}
\label{sec:comp-corr}
In this section, we will show how to explicitly go from a matrix $\A\in M_n(R)$ to the matrix of the corresponding extension in $E_{k[t]}(M^A,M^A)$ (as defined in Section~\ref{sec:bases-extensions}) and back.
By the primary decomposition, we may assume that the image $A$ of $\A$ in $M_n(k)$ is $p$-primary for some irreducible polynomial $p\in k[t]$.
Further, we may assume that $A$ is in primary rational normal form corresponding to some fixed partition $\lambda$.
\subsection{The extension corresponding to a matrix}
\label{sec:extens-from-matr}
Consider a matrix $\A\in M_n(R)$, whose image in $M_n(k)$ is (for some $p\in \Irr kt$) 
\begin{equation*}
  A = C_{p^{\lambda_1}}\oplus \dotsb \oplus C_{p^{\lambda_r}}.
\end{equation*}
Take $\bp$ to be a monic lift of $p$ to $R[t]$. 
We may think of $R^n$ as 
\begin{equation}
  \label{eq:M}
  R[t]/\bp^{\lambda_1}\oplus\dotsb\oplus R[t]/\bp^{\lambda_r},
\end{equation}
and regard $\A$ as an $R$-linear endomorphism of this space.
As in the proof of Lemma~\ref{lemma:main}, let
\begin{equation*}
  \A_0 = C_{\bp^{\lambda_1}}\oplus \dotsb \oplus C_{\bp^{\lambda_r}}.
\end{equation*}

Thus we have two extensions
\begin{gather}
  \tag{$\xi_\A$}
  \ses{M^A}{M^{\A}}{M^A}\pi p
  \\
  \tag{$\xi_{\A_0}$}
  \ses{M^A}{M^{\A_0}}{M^A}\pi{p_0}.
\end{gather}
The extension $\xi\in E_{k[t]}(M^A,M^A)$ such that $\A=\A_\xi$ (see Definition~\ref{defn:class}) is given by
\begin{equation}
  \label{eq:xi}
  \iota(\xi)=\xi_\A -\xi_{\A_0}.
\end{equation}
The prescription for determining the difference of these two extensions is to take the pre-image of the diagonal copy of $M^A$ in $M^A\oplus M^A$ in $M^\A\oplus M^{\A_0}$, and set $\pi x\oplus \pi x=0$ for all $x\in M^A$. 

Write $h_1,\dotsc,h_l$ and $h_1',\dotsc,h_l'$ for the generators of two copies of the $R$-module in (\ref{eq:M}).
Thus
\begin{align*}
  \bp^{\lambda_j}(h_j\oplus h'_j) & = \bp(\A)^{\lambda_j}h_j\oplus \bp(\A_0)^{\lambda_j}h'_j\\
  & = \bp(\A)^{\lambda_j}h_j\oplus 0.
\end{align*}
Since
\begin{equation*}
  \bp(\A)^{\lambda_j} h_j \equiv 0 \mod \pi M^\A,
\end{equation*}
there exist polynomials $\epsilon_{ij}\in R[t]$ such that 
\begin{equation}
  \label{eq:eps}
  \bp(\A)^{\lambda_j} h_j = \sum_{i=1}^r \eps_{ij} \pi h_j,
\end{equation}
Comparing (\ref{eq:eps}) with (\ref{eq:relation}) gives the desired extension matrix corresponding to $\A$:
\begin{theorem} If $E(\xi)=(\epsilon_{ij})$, where the polynomials $\epsilon_{ij}\in R[t]$ satisfy (\ref{eq:eps}), then $\A = \A_\xi$.
\end{theorem}

\subsection{Nilpotent Case}

Now suppose that $p=t$, and $\A_0=J_{\lambda}$, meaning the nilpotent lower triangular Jordan matrix corresponding to the partition $\lambda$.  Write $\A=\A_0+ \pi X$ as usual.  

Let $\lambda=(\lambda_1, \ldots, \lambda_r)$.  For each $1 \leq \iota \leq r$, we must compute 
\begin{equation*} \label{Pune}
  \begin{split}
    \A^{\lambda_{\iota}}h_{\iota} &= (J_\lambda+ \pi X)^{\lambda_\iota} h_{\iota} \\
    &= J_\lambda^{\lambda_\iota} h_{\iota} + \pi \left(\sum_{a+b=\lambda_{\iota}-1} J_\lambda^a X J_\lambda^b \right) h_{\iota} \\
  \end{split}
\end{equation*}
For $N \geq 1$,  let us write $\fL_N: M_n(k) \to M_n(k)$ for the linear map defined by
\begin{equation} \label{yep}
  \fL_N(X)= \sum_{a+b=N-1} J_\lambda^a X J_\lambda^b.
\end{equation}
Then the above gives
\begin{equation*}
  \A^{\lambda_{\iota}}h_{\iota} = \pi \fL_{\lambda_{\iota}}(X)h_\iota.
\end{equation*}

We will evaluate $\fL_N$ on a basis of $M_n(k)$ given by the usual coordinate matrices, which we organize according to blocks.  We view a matrix $B \in M_n(k)$ as an $r \times r$ block matrix with $(\alpha,\beta)$-th block $B(\alpha,\beta) \in M_{\lambda_\alpha \times \lambda_\beta}(k)$.

For each $\alpha$, write $\pr_\alpha: k^n \to k^{\lambda_\alpha}$ for the projection to the components corresponding to $\lambda_\alpha$.  Note that 
\begin{equation*}
  \pr_\alpha(h_\iota)= \begin{cases} e_1 & \text{ if $\alpha =\iota$} \\
    0 & \text{otherwise} 
  \end{cases},
\end{equation*}
and that $\pr_\alpha(B h_\beta)=B(\alpha,\beta)e_1$.

Regarding (\ref{yep}) as a block matrix, we have
\begin{equation*}
  \fL_N(X)(\alpha,\beta)= \sum_{a+b=N-1} J_{\lambda_\alpha}^a X(\alpha,\beta) J_{\lambda_\beta}^b \in M_{\lambda_\alpha,\lambda_\beta}(k),
\end{equation*}
where generally $J_{i}$ is the nilpotent lower triangular single Jordan block of size $i$.  

We are thus led to study the linear maps
$L_N: M_{p,q}(k) \to M_{p,q}(k)$
defined by
\begin{equation*}
  L_N(Y)=\sum_{a+b=N-1} J_p^a Y J_q^b.
\end{equation*}

By linearity, it is enough to determine $L_N$ on a basis.  We pick the standard basis $Y_{ij} \in M_{p,q}(k)$, where $Y_{ij}$ is the $p \times q$ matrix consisting of a $1$ in the $(i,j)$ entry and $0$ elsewhere.
It is easy to calculate that $J_p^a Y_{ij} J_q^b=Y_{i+a,j-b}$, if we set the convention  that $Y_{ij}=0$ if $i>p$ or $j \leq 0$.  Therefore
\begin{equation*}
  L_N(Y_{ij})=\sum_{a+b=N-1} Y_{i+a,j-b}.
\end{equation*}
In particular we have 
\begin{equation*}
  L_N(Y_{ij})e_1= \begin{cases} e_{N+(i-j)} & \text{ if $1 \leq N+(i-j) \leq p$} \\
    0 & \text{otherwise} 
  \end{cases}. 
\end{equation*}
Let us return to the study of $\fL_{\lambda_\iota}(X)h_{\iota}$.  We consider the following basis of $M_n(k)$: \newline Pick $1 \leq \alpha_0, \beta_0 \leq r$, $1 \leq i \leq \lambda_{\alpha_0}$, and $1 \leq  j \leq \lambda_{\beta_0}$.
Write  $X_0=X_{\alpha_0,\beta_0,i,j} \in M_n(k)$ for the block matrix so that $X(\alpha,\beta)=0$ unless $\alpha=\alpha_0$ and $\beta=\beta_0$, and moreover that $X(\alpha_0,\beta_0)=Y_{ij}$.
We have
\begin{equation*}
  \pr_\alpha (\fL_{\lambda_\iota}(X_0)h_\iota)=L_{\lambda_\iota}(X_0(\alpha,\iota))e_1.
\end{equation*} 
This is equal to $0$, unless $\alpha=\alpha_0$ and $\iota=\beta_0$, in which case, it is equal to $L_{\lambda_\iota}(Y_{ij})e_1$.  The latter is equal to $0$ unless $i \leq j$, in which case
it is equal to $e_{\lambda_\iota+(i-j)} \in k^{\lambda_\alpha}$.  Thus,

\begin{equation*}
  \pr_\alpha (\fL_{\lambda_\iota}(X_0)h_\iota)=  \begin{cases} e_{\lambda_\iota+(i-j)} & \text{ if $\alpha=\alpha_0$, $\iota=\beta_0$, and $i \leq j$} \\
    0 & \text{otherwise} 
  \end{cases}. 
\end{equation*}

Thus we obtain
\begin{equation*}
  \fL_{\lambda_\iota}( X_{\alpha_0,\beta_0,i,j}) h_\iota= \begin{cases} t^{\lambda_{\iota}+(i-j)-1} h_{\alpha_0} & \text{ if $i \leq j$ and $\iota=\beta_0$} \\
    0 & \text{otherwise} 
  \end{cases},
\end{equation*} 
and so 
\begin{equation*}
  \eps_{\iota,\kappa}(X_{\alpha_0,\beta_0,i,j})= \begin{cases} t^{\lambda_{\iota}+(i-j)-1}  & \text{ if $i \leq j$, $\alpha_0=\kappa$ and $\beta_0=\iota$} \\
    0 & \text{otherwise} 
  \end{cases}.
\end{equation*}
Note that the expression only depends on $i,j$ through the difference $i-j$.

We extend this formula by linearity to all of $M_n(k)$.  Please see Table 2 for some examples. 

\begin{table}
  \caption{Extensions from Matrices: $p(t)=t$}
  \label{tab:exts}
  \tiny
      \centering
      \begin{tabular}{cccc}
        \hline 
        $\lambda$ & $\A_0$ & $X$ & $E(\A)=E(\A_0+\pi X)$\\
        \hline
        $(2,1)$ & $\Mat 000100000$ &  $\Mat abcdefghi$ & $\mat {(a+e)t+b}chi $ \\[0.5cm]
        $(3,1)$ & $\left(\begin{array}{cccc}
            0   & 0 &  0 & 0 \\
            1  &  0 & 0 & 0  \\
            0  & 1  & 0  & 0 \\
            0   & 0 &  0 & 0 \\
          \end{array} \right)$ & $\left(\begin{array}{cccc}
            a   & b &  c & d \\
            e  &  f & g & h  \\
            i  & j  & k  & l \\
            m  & n &  o & p \\
          \end{array} \right)$ & $\mat {(a+f+k)t^2+(b+g)t+c}dop$ \\[1cm]
        $(2,2)$ & $\left(\begin{array}{cccc}
            0   & 0 &  0 & 0 \\
            1  &  0 & 0 & 0  \\
            0  & 0  & 0  & 0 \\
            0   & 0 &  1& 0 \\
          \end{array} \right)$ & $\left(\begin{array}{cccc}
            a   & b &  c & d \\
            e  &  f & g & h  \\
            i  & j  & k  & l \\
            m  & n &  o & p \\
          \end{array} \right)$ & $\mat {(a+f)t+b}{(c+h)t+d}{(n+i)t+j}{(k+p)t+l}$ \\[1cm]
        $(2,1,1)$ & $\left(\begin{array}{cccc}
            0   & 0 &  0 & 0 \\
            1  &  0 & 0 & 0  \\
            0  & 0  & 0  & 0 \\
            0   & 0 &  0 & 0 \\
          \end{array} \right)$ & $\left(\begin{array}{cccc}
            a   & b &  c & d \\
            e  &  f & g & h  \\
            i  & j  & k  & l \\
            m  & n &  o & p \\
          \end{array} \right)$ & $\Mat {(a+f)t+b}cdjklnop$ \\[1cm]
        \hline
      \end{tabular}
\end{table}
\subsection{General Formula}
The reader will note that the coefficients of the polynomial entries of $E(\A)$ in Table 2 are given by traces of the blocks of $X$, and also traces ``above the diagonal''.
In this section we will give a more direct formula for $E(\A)$ in terms of these more general traces.

\begin{defn} Let $A \in M_{m \times n}(k)$, and $0 \leq \ell < n$.  If $A=(a_{ij})$, we define 
  \begin{equation}
    \tr_\ell(A)= \sum_{i=1}^{\min(m,n-\ell)} a_{i,i+\ell}.
  \end{equation}
\end{defn}

In words, $\tr_\ell(A)$ is the sum of the entries on the diagonal of $A$ which is $\ell$ to the right of the main diagonal.  In particular, $\tr_0(A)=\tr(A)$ is the usual trace when $A$ is a square matrix.

\begin{defn} Let $m,n$ be positive integers.
  We define a map $\varphi=\varphi_{m,n}: M_{m \times n}(k) \to k[t]/t^{\min(m,n)}$ by
  \begin{equation}
    \varphi(A)= \sum_{i=1}^{\min(m,n)} \tr_{n-i}(A) t^{i-1}.
  \end{equation}
\end{defn}

\begin{theorem}
  \label{prop:general-formula}
  The extension matrix associated to $\A=J_{\lambda}+ \pi X$ is the matrix $E(\A) \in E_{\lambda}$ with
  \begin{equation}
    E(\A)(\alpha,\beta)=\varphi_{\lambda_\alpha,\lambda_\beta}(X(\alpha,\beta)).
  \end{equation}
\end{theorem}

\begin{proof} Since both sides are $k$-linear, it is enough to show that both sides agree at the elements $X_{\alpha_0,\beta_0,i,j}$ from the previous section.  This is straightforward.
\end{proof}

\subsection{Matrices from extension classes}
\label{sec:normal-forms-matrices}

Let $A \in M_n(k)$.  
Suppose that $p \in \irr k$ is of degree $d$, and $p(A)^n=0$; in other
words, that $A$ is $p$-primary.   
Then there exists a partition $\lambda$ such that $A$ is similar to
$C_\lambda(p)$.
Let us see how, given $\xi\in E_{k[t]}(M^A,M^A)$, to write down a
matrix in $M_n(R)$ which lies in the similarity class corresponding to
$\xi$ under Theorem~\ref{theorem:main}. 

Write $(e_i)$ for a $k$-basis of the first $M^A$ and $(f_i)$ for a $k$-basis of the second $M^A$.  
Let $\xi\in E_{k[t]}(M^A,M^A)$ be given by
\begin{equation}
  \label{eq:7}
  \tag{$\xi$}
  \xses\xi{M^A}{E_\xi}{M^A}{\iota}{}
\end{equation}
with $E(\xi)=(\epsilon_{ij})$.
One has lifts $(\tilde f_i)$ in $E_\xi$ of the $(f_i)$, and $E_\xi$ is defined (as a $k[t]$-module) by generators 
\begin{equation*}
  \iota(e_1),\dots, \iota(e_l),\tilde f_1,\dotsc,\tilde f_l
\end{equation*}
subject to the relations (for $1\leq j\leq l$)
\begin{gather}
  p^{\lambda_j} e_j = 0,\\
  p^{\lambda_j}\tilde f_j = \sum_{i=1}^l \epsilon_{ij}\iota(e_i) 
\end{gather}
If we wish to think of these as generators and relations for
$R[t]$-modules, then we must add the relations $\pi\iota(e_j)=0$ and
$\pi\tilde f_j=0$ for $1\leq j\leq l$.

For the extension $\tilde \xi_0$ of $R[t]$-modules given by
\begin{equation*}
  \xses{\xi_0}{M^A}{M^{\A_0}}{M^A}\iota {},
\end{equation*}
$M^{\A_0}$ has generators
$g_1,\dotsc,g_l$ with relations
\begin{equation}
  \label{eq:8}
  \bp^{\lambda_j} g_j=0.
\end{equation}
Here $\bp \in R[t]$ is the monic lift of $p$ as in the proof of Lemma~\ref{lemma:main}.
Recall that, in general, given two extensions
\begin{gather*}
  \ses{N}{E_1}{M}{}{}\\
  \ses{N}{E_2}{M}{}{},
\end{gather*}
their Baer sum is the extension obtained
by taking the pre-image of the diagonal copy of $M$ inside $M\oplus M$
under the map $E_1\oplus E_2\to M\oplus M$ and identifying the
images of the two copies of $N$ under $N\oplus N\to E_1\oplus E_2$.

Applying this observation to construct $\iota(\xi)+\tilde\xi_0$ shows
that in the extension
\begin{equation*}
  \xses{\iota(\xi)+\tilde\xi_0}{M^A}{\tilde E_\xi}{M^A}\iota {},
\end{equation*}
$\tilde E_\xi$ is generated by $\iota(e_1),\dotsc,\iota(e_l)$ and
$h_1,\dotsc,h_l$, where $h_j=\tilde f_j+g_j$.
Since the images of $M^A$ in $E_\xi$ and $M^{\A_0}$ are identified in
$\tilde E_\xi$, it follows that $i(e_j)=\pi g_j=\pi(\tilde
f_j+g_j)=\pi h_j$ for $1\leq j\leq l$.
This allows us to dispense with the generators $\iota(e_j)$.
The relations on the generators $h_1,\dotsc,h_l$ are $\bp^{\lambda_j}h_j=0$.

This means that, as an $R$-module, $\tilde E_\xi$ has basis
\begin{gather*}
  h_1, t h_1, \dotsc, t^{d \lambda_1-1}h_1,\\
  h_2, t h_2, \dotsc, t^{d \lambda_2-1}h_2,\\
  \vdots\\
  h_l, t h_l,\dotsc,t^{d \lambda_r-1}h_l.
\end{gather*}
Then $\A_\xi$ is the matrix of multiplication by $t$ with
respect to this basis.

We have
\begin{equation}
  t \cdot t^{d\lambda_j-1} h_j= \bp^{\lambda_j} h_j- \sum_{m=0}^{d-1}c_mt^mh_j,
\end{equation}
where $\bp^{\lambda_j}=t^{d \lambda_j}+\sum_{m=0}^{d-1}c_mt^m$.
Now 
\begin{equation}
  \begin{split}
    \bp^{\lambda_j}h_j &= \bp^{\lambda_j}f_j \\
    &=  \pi \sum_i \epsilon_{ij} h_i, \\
  \end{split}
\end{equation}
where 
\begin{equation*}
  \epsilon_{ij}=\epsilon_{ij}^0+\epsilon_{ij}^1t +\dotsb+ \epsilon_{ij}^{d \lambda_j-1} t^{d \lambda_j-1}
\end{equation*}
is reduced modulo $p^{\lambda_j}$.
We have proved
\begin{theorem}
  \label{theorem:matrix-from-extension}
  Suppose that $E(\xi)=(\epsilon_{ij})$.
  Then
  \begin{equation*}
    \A_\xi=\A_0+\pi \phi(\epsilon),
  \end{equation*}
  where $\phi(\epsilon)$ is the $l\times l$ block matrix with $(i,j)$th
  block of size $(d \lambda_i)\times (d \lambda_j)$ given by
  \begin{equation} \label{express}
    \phi(\epsilon)_{ij}=
    \begin{pmatrix}
      0 & \cdots & 0 & \epsilon_{ij}^0\\
      0 & \cdots & 0 & \epsilon_{ij}^1\\
      \vdots & & \vdots & \vdots\\
      0 & \cdots & 0 & \epsilon_{ij}^{d\lambda_i-1}
    \end{pmatrix}
    .
  \end{equation}
\end{theorem}
\section{Conjugacy and Transposition}
\label{sec:conj-transp}
It is well-known that in the ring of matrices over a field, a matrix $A$ is always similar to its transpose $A^T$.  As we will see, this is not true in $M_n(R)$ for $n \geq 3$.  In this section, as a general application of the machinery in this paper, we will study the question of when $\A$ is similar to its transpose.  First we will show that if $A$ is nilpotent, then the answer can be read off easily from the extension matrix $E(\A)$.
Next we will use a stability property of similarity classes in $M_n(R)$ to reduce the general case to the nilpotent case.
We will use these results to describe exactly which similarity classes in $M_n(R)$ are self-transpose for $n\leq 4$.

Every matrix $A\in M_n(R)$ can be written as a sum $\oplus_p \A_p$ as in Remark~\ref{remark:primary}.
By the uniqueness (up to similarity) of primary decomposition, $(\A^T)_p$ is similar to $(\A_p)^T$.
Thus, $\A$ is similar to $\A^T$ if and only if $\A_p$ is similar to $\A_p^T$ for every $p$.
Thus the problem of characterizing self-transpose classes in $M_n(R)$ is reduced to the primary case, where our main result is:
\begin{theorem} 
  \label{theorem:transposes}
  Let $A$ be a $p$-primary matrix in rational normal form.
  For each $\xi\in E_{k[t]}(M^A,M^A)$, $\A_\xi$ is similar to its transpose if and only if $E(\xi)$ lies in the same $G_A$-orbit as $E(\xi)^T$.
\end{theorem}
Note that when $A$ is $p$-primary and in rational normal form, then $M^A=M_\lambda(p)$ for some partition $\lambda$.
In this context, $E(\xi)$ is the matrix corresponding to $\xi$ constructed in Section~\ref{sec:bases-extensions}.

\subsection{Nilpotent Case}
\begin{defn} Let $X \in M_{m \times n}(k)$.  Write $X^T \in M_{n \times m}(k)$ for the transpose of $X$, and ${}^TX \in M_{n \times m}$ for the ``transpose along the antidiagonal" defined by
  \begin{equation*}
    ({}^TX)_{i,j}=X_{m+1-j,n+1-i}.
  \end{equation*}
\end{defn}

Let $w_n \in \GL_n(k)$ is the permutation matrix comprised of $1$s down the antidiagonal.  We have ${}^T X=w_n X^Tw_m^{-1}$.

We record the following immediate calculations.

\begin{lemma} \label{transpose lemma}  \label{phi and transpose} 
  \begin{enumerate}
  \item If $\max(0,m-n) \leq \ell < \min(m,n)$, then
    \begin{equation*}
      \tr_{\ell}({}^T X)=\tr_{\ell+(n-m)}(X).
    \end{equation*}
  \item For $X \in M_{m \times n}(k)$, we have
    \begin{equation*}
      \varphi_{n,m}({}^T X)=\varphi_{m,n}(X).
    \end{equation*}
  \end{enumerate}
\end{lemma}

\begin{theorem} 
  \label{theorem:nilpotent-transpose}
  Let $\A \in M_n(R)$ be nilpotent.  Then $\A$ is similar to its transpose if and only if $E(\A)$ is equivalent to $E(\A)^T$.
\end{theorem}

\begin{proof}
  Note that if $w$ is the block diagonal matrix $\diag(w_{\lambda_1}, \ldots, w_{\lambda_r})$, then $J_{\lambda}^T=w J_{\lambda} w^{-1}$, and so
  \begin{equation*}
    wA^Tw^{-1}= J_\lambda+ \pi w X^Tw^{-1}.
  \end{equation*}

  Applying Theorem~\ref{prop:general-formula}, we have $E(\A^T) \sim E(w\A^T w^{-1})$, but the latter is equal to $E(\A)^T$ by Lemma~\ref{phi and transpose}.
\end{proof}

\subsection{Stability of similarity} 
In order to reduce the general case to the nilpotent case, we will need the fact that when $R$ is contained in a larger ring $S$, then similarity over $R$ is equivalent to similarity over $S$.
This result is interesting in its own right, and is proved in the greater generality of Artinian rings.    
\begin{theorem}[Stability of similarity] \label{Rationality} Let $R$ be an Artinian ring, and $S$ an extension ring of $R$ which is free and finitely generated as an $R$-module.  If two elements $\A,\A' \in M_n(R)$ are conjugate when regarded as matrices in $M_n(S)$, then they are already conjugate in $M_n(R)$.
\end{theorem}
\begin{proof}
  Note that since $R$ is Artinian, $M^{\A}$ has finite length as an $R$-module, and hence as an $R[t]$-module.

  By hypothesis we have
  \begin{equation*}
    M^{\A} \otimes_{R[t]} S[t] \cong M^{\A'} \otimes_{R[t]} S[t]
  \end{equation*}
  as $S[t]$-modules, and therefore as $R[t]$-modules.

  Write $d$ for the cardinality of a basis of $S$ over $R$. Then the above equation gives
  \begin{equation*}
    (M^{\A})^d \cong  (M^{\A'})^d
  \end{equation*}
  as $R[t]$-modules.  By the Krull-Schmidt Theorem, this implies that $M^{\A} \cong  M^{\A'}$, as desired.
\end{proof}

\subsection{The general case}

We start with the easy observation that if $\A$ is conjugate to its transpose and $\A'$ is conjugate to $\A$, then $\A'$ is conjugate to its transpose.

Let $\A \in M_n(R)$.  By Remark~\ref{remark:primary}, $\A$ is conjugate to a sum $\oplus_p \A_p$ with $p\in \irr k$ and $r(\A_p)$ a $p$-primary matrix.
Since $r(\A_p^T)$ is again $p$-primary, it follows that $\A$ is conjugate to its transpose if and only if $\A_p$ is conjugate to its transpose for each $p\in \irr k$.
The problem therefore reduces to the case where $A=r(\A)$ is primary. 
Then we may as well assume $A$ is $C_{\lambda}(p)$ for some $p\in \irr k$ and some partition $\lambda$. 

Consider a splitting field $k'$ of $p$ over $k$.
Let $R'$ be the corresponding unramified extension of $R$
(for example, if $R=k[t]/t^2$, then $R'=k'[t]/t^2$ and if $R=W_2(k)$, take $R'=W_2(k')$).   

From Theorem~\ref{Rationality} (stability of similarity) we have $\A$ is conjugate to $\A^T$ in $M_n(R)$ if and only if $\A$ is conjugate to $\A^T$ in $M_n(R')$.
Regarded as a matrix in $M_n(k')$, $A$ is similar to its Jordan normal form $\oplus_\alpha J_{\lambda}(\alpha)$.  Once again, primary decomposition reduces us to the case where $A=J_{\lambda}(\alpha)$, and then by subtracting $\alpha$, we have reduced to the nilpotent case.

Here are some consequences.

\begin{cor}
  \label{cor:reg-ss}
  Let $A \in M_n(k)$.
  If $A$ is either regular or semisimple, then any $\A\in M_n(R)$ with $r(\A)=A$ is similar to its transpose.
\end{cor}
\begin{proof}
  We recall that $A\in M_n(k)$ is regular if, for each irreducible polynomial $p\in k[t]$, the $p$-primary part $A_p$ of $A$ is of the from $C_\lambda(p)$ where $\lambda$ is a partition with only one part, and is semisimple if, for each irreducible polynomial $p\in k[t]$, $A_p$ is of the form $C_\lambda(p)$ with $\lambda = (1^k)$ for some positive integer $k$.
  In the regular case, the extension matrix is its own transpose.
  In the semsimple case, the space $E_\lambda(p)$ is isomorphic to $M_r(k)$ (see Theorem~\ref{theorem:rectangular} below), and is therefore each element lies in the same orbit as its transpose.
  Thus Theorem~\ref{theorem:nilpotent-transpose} implies the corollary in both cases.
\end{proof}

Let us now turn to small-dimensional examples.  For $n=2$, then the only partitions $\lambda$ that can appear in the Jordan decompositions of the primary parts of a matrix are partitions of $1$ and $2$.
These are all covered in the proof of Corollary~\ref{cor:reg-ss} above.
We obtain:
\begin{theorem} Every matrix $\A \in M_2(R)$ is conjugate to its transpose.
\end{theorem}
And immediately:
\begin{cor} If $A=C_{\lambda}(p)$ with $\lambda=(2,2)$, then any lift of $A$ is conjugate to its transpose.
\end{cor}

For the partition $\lambda=(2,1)$, non-self-transpose classes exist (see the tables from Section \ref{m+1,1_section}).
For $n\geq 3$, there always exist functions $c\in C(n)$ (see Section~\ref{sec:analysis-of-fibres}) where $c(p)=(2,1)$ for some $p\in \irr k$.
It follows that:
\begin{theorem}
  For all $n\geq 3$, there exist matrices $\A\in M_n(R)$ which are not similar to their transposes.
\end{theorem}

By using the tables in Section~\ref{sec:description-orbits}, we will be able to determine exactly how many classes over each similarity class in $M_n(k)$ are self-transpose.

\section{Description of orbits}
\label{sec:description-orbits}
By Theorems~\ref{theorem:main} and~\ref{theorem:centralizer}, in order to understand the number of similarity classes in $M_n(R)$ which lie above the class of $A\in M_n(k)$ along with their centralizers, it suffices to understand the $G_A$-orbits in $E_{k[t]}(M^A,M^A)$ along with their centralizers.
Writing representatives for these orbits will allow us to write down matrices representing these similarity classes by Theorem~\ref{theorem:matrix-from-extension}.
These computations are carried out for several special cases of $A$ in this section.
In the next section, these computations will be used to describe the similarity classes in $M_3(R)$ and $M_4(R)$.
\subsection{The cyclic case}
Recall that a matrix $A\in M_n(k)$ is said to be regular if $M^A$ is a
cyclic $k[t]$-module.
\begin{theorem}
  \label{theorem:regular}
  If $A$ is regular, then $E_{k[t]}(M^A,M^A)=M^A$, and the
  $G_A$-action on $E_{k[t]}(M^A,M^A)$ is trivial.
  Consequently, the similarity classes in $M_n(R)$ that lie above $A$
  are in bijective correspondence with the elements of $M^A$. Moreover, for each
  $\A\in M_n(R)$ such that $r(\A)=A$, $\bar G_\A=G_A$.
\end{theorem}
\subsection{The elementary case}
\begin{theorem}
  \label{theorem:elemenary}
  For any $p\in \Irr kt$ and $\lambda=(1^n)$, $G_\lambda(p)$ is
  isomorphic to $\GL_n(K)$ and $E_\lambda(p)$ is
  isomorphic to $M_n(K)$ as a $G_\lambda(p)$-set, where $K=k[t]/p$ and
  $g\in \GL_n(K)$ acts on $E\in M_n(K)$ by conjugation: $E\mapsto gEg\inv$.
\end{theorem}
\subsection{The rectangular case}
\begin{theorem} 
  \label{theorem:rectangular}
  For any $p\in \Irr kt$ and $\lambda=(m^n)$, $G_\lambda(p)$ is
  isomorphic to $\GL_n(k[t]/p^m)$ and $E_\lambda(p)$ is isomorphic to
  $M_n(k[t]/p^m)$ as a $G_\lambda(p)$-set where $g\in \GL_n(k[t]/p^m)$
  acts on $E\in M_n(k[t]/p^m)$ by conjugation: $E\mapsto gEg\inv$.
\end{theorem}
\begin{cor}
  Let $k$ be a finite field.
  The following problems are equivalent:
  \begin{enumerate}
  \item \label{item:2}Enumerating the set of similarity classes in $M_n(R)$ that lie above a given similarity class in $M_n(k)$ along with the cardinalities of their centralizers for every local principal ideal ring $R$ of length two with residue field $k$ and every positive integer $n$.
  \item \label{item:1}Enumerating the set of similarity classes in $M_n(K[t]/t^m)$ along with the cardinalties of their centralizers for all positive integers $m$ and $n$ and every finite extension $K$ of $k$.
  \end{enumerate}
\end{cor}
\begin{proof}
  The enumeration of similarity classes in $M_n(k[t]/t^2)$ along with the cardinalities of their centralizers is clearly a special case of \ref{item:1}. By Theorem~\ref{theorem:comparison}, this is equivalent to the enumeration of similarity classes in $M_n(R)$ along with their centralizers for any local principal ideal ring of length two and residue field $R$.
  Thus a solution to \ref{item:1} implies a solution to \ref{item:2}.

  Let $K$ be any finite extension of $k$ of degree $d$.
  Then there is an irreducible polynomial $p\in \Irr kt$ such that $K=k[t]/p$.
  By Theorem~\ref{theorem:rectangular}, the enumeration of similarity classes of matrices that lie above the class of a matrix $A\in M_{mnd}(k)$ for which
  \begin{equation*}
    M^A=(k[t]/p^m)^{\oplus n}
  \end{equation*}
  is equivalent to the determination of similarity classes of matrices in $M_n(k[t]/p^m)$.
  Since $k$ is a perfect field, the rings $k[t]/p^m$ and $K[t]/t^m$ are isomorphic.
  It follows that a solution to \ref{item:2} implies  a solution to \ref{item:1}.
\end{proof}

\subsection{$\lambda=(m+1,1)$}
\label{m+1,1_section}
Consider the case where, for some $m\geq 1$, $\lambda=(m+1,1)$.
Then
\begin{equation*}
  M_{(m+1,1)}(p)=k[t]/p^{m+1}\oplus k[t]/p.
\end{equation*}
By Theorem~\ref{theorem:extension-matrix},
\begin{equation*}
  E_{(m+1,1)}(p)=\left\{\mat xyzw \;|\; x\in k[t]/p^{m+1},\;y,z,w\in k[t]/p\right\}.
\end{equation*}
Recall from the discussion at the beginning of Section~\ref{sec:matr-transf} that
\begin{equation*}
  G_{(m+1,1)}(p)=\left\{\mat a{p^mb}cd\;|\; a\in k[t]/p^{m+1},\;b,c,d\in k[t]/p\right\}.
\end{equation*}

\begin{theorem}
  \label{theorem:m+1,1}
  A complete set of representatives for the $G_{(m+1,1)}(p)$-orbits in $E_{(m+1,1)}(p)$ is given by
  \begin{enumerate}
  \item \label{item:3} $\mat x00w$, with $x\in k[t]/p^{m+1}$ such that $p|x$, and $w\in k[t]/p$.
    Each of these elements has isotropy group $G_{(m+1,1)}(p)$.
  \item \label{item:4} $\mat x00w$, with $x\in (k[t]/p^{m+1})^*$ and $w\in k[t]/p$.
    Each of these elements has isotropy group
    \begin{equation*}
      \left\{ \mat a00d\;|\;a\in (k[t]/p^{m+1})^*,\;d\in (k[t]/p)^*\right\}.
    \end{equation*}

  \item \label{item:5} $\mat x1zw$, with $x\in k[t]/p^m$ such that $p|x$, $z,w\in k[t]/p$.
    Each of these elements has isotropy group
    \begin{equation*}
      \left\{\mat a{p^mb}{bz} a\;|\;a\in (k[t]/p^{m+1})^*,\; b\in k[t]/p\right\}.
    \end{equation*}

  \item \label{item:6} $\mat x01w$, with $x\in k[t]/p^m$ such that $p|x$, $w\in k[t]/p$.
    Each of these elements has isotropy group
    \begin{equation*}
      \left\{\mat a0c a\;|\;a\in (k[t]/p^{m+1})^*,\; c\in k[t]/p\right\}.
    \end{equation*}
  \end{enumerate}
\end{theorem}

\begin{proof}
  The Birkhoff moves give rise to the transformations
  \begin{gather*}
    \xymatrix{ 
      {\mat xyzw} \ar[r]^{\hspace{-1.1cm}(E1)}&  {\mat{x+p^m\alpha z}{y-\alpha x}z{w-\alpha z}}
    }
    \\
    \xymatrix{
      {\mat xyzw} \ar[r]^{\hspace{-1.1cm}(E2)}& {\mat{x-p^m\alpha y}y{z+\alpha x}{w+\alpha y}}
    }
    \\
    \xymatrix{
      {\mat xyzw} \ar[r]^{\hspace{-0.5cm}(E3)} & {\mat x{\alpha y}{\alpha\inv z}w}
    }
  \end{gather*}
  If $(x,p)=1$, then (\ref{eq:12}) and (\ref{eq:13}) can be used to reduce $\mat xyzw$ to the form \ref{item:4}.

  Now suppose that either $(y,p)=1$ or $(z,p)=1$.
  Then (\ref{eq:12}) and (\ref{eq:13}) can be used to modify $x$ to a representative modulo $p^m$.
  If $(y,p)=1$, then (\ref{eq:14}) can be used to reduce $\mat xyzw$ to the form \ref{item:5}.
  On the other hand, if $p|y$ and $(z,p)=1$, (\ref{eq:14}) can be used to reduce $\mat xyzw$ to the form \ref{item:6}.

  The only remaining case is when $\mat xyzw$ is of the form \ref{item:3}.
  It follows that $\mat xyzw$ can be reduced to one of the four forms in the statement of the theorem.

  The isotropy calculations are straightforward; $\mat a{p^mb}cd$ stabilizes $\mat xyzw$ if and only if
  \begin{equation*}
    \mat a{p^mb}cd \mat xyzw = \mat xyzw \mat ab{p^mc}d.
  \end{equation*}
  Substituting each of the canonical froms \ref{item:3}-\ref{item:6} for $\mat xyzw$ yields the description of isotropy subgroups in the statement.
  The isotropy of the type \ref{item:3}, being all of $G_A$, is not conjugate to that of any of the other types, showing that elements of type \ref{item:3} are not in the same orbit as elements of any other type.
  Since the Birkhoff moves (\ref{eq:12})-(\ref{eq:14}) do not change the residue of $x$ modulo $p$, the types \ref{item:4} are not conjugate to any of the other types.

  If $p|x$, the Birkhoff moves (\ref{eq:12})-(\ref{eq:14}) can only scale $y$ and $z$ by $\alpha$ which is coprime to $p$.
  Therefore elements of type \ref{item:5} and \ref{item:6} can not be in the same orbit.

  It remains to show that within types \ref{item:3} and \ref{item:4}, matrices with different representatives are in distinct orbits.
  Now $\mat x00w$ and $\mat {x'}00{w'}$ are in the same $G_A$-orbit if and only if
  \begin{equation*}
    \mat a{t^mb}cd\mat x00w=\mat{x'}00{w'}\mat ab{t^mc}d
  \end{equation*}
  for some $a,b,c,d\in k[t]$ with $(a,p)=(d,p)=1$.
  In other words,
  \begin{equation*}
    \mat{ax}00{dw}=\mat{ax'}00{dw'}
  \end{equation*}
  whence we may conclude that $x\equiv x'\mod p^{m+1}$ and $w\equiv w'\mod p$, showing that matrices of types \ref{item:3} and \ref{item:4} with different representatives are in different orbits.
\end{proof}
When $k[t]/p$ is a finite field of order $q$, one can easily read off the cardinalities of orbits with a given centralizer cardinality from Theorem~\ref{theorem:m+1,1}.
Moreover, it is not hard to see which of these classes are self-transpose.
The results are collected in Table~\ref{tab:m1}.
\begin{table}
  \centering
  \caption{$\lambda = (m+1,1)$}
  \label{tab:m1}
  \begin{equation*}
    \begin{array}{|c|c|c|c|}
      \hline
      \text{Reference} &\text{Centralizer} & \text{Total Classes} & \text{Self-transpose}\\
      \hline
      \text{\ref{item:3}}& q^{m+4}(1-q\inv)^2 & q^{m+1} & \text{ all }\\
      \hline
      \text{\ref{item:4}}&q^{m+2}(1-q\inv)^2 & q^{m+2}-q^{m+1} & \text{ all } \\
      \hline
      \begin{matrix}
        \text{\ref{item:5}}\\
        \text{\ref{item:6}}
      \end{matrix}
      & q^{m+2}(1-q\inv) & q^{m+1}+q^m & q^{m+1}-q^m\\
      \hline
      \hline
      & \text{ Totals: } & q^{m+2}+q^{m+1}+q^m & q^{m+2}+q^{m+1}-q^m\\
      \hline
    \end{array}
  \end{equation*}
\end{table}

\subsection{$\lambda = (m+1,1,1)$} 
Now consider the case
\begin{equation*}
  M_{(m+1,1,1)}(p)=k[t]/p^{m+1}\oplus k[t]/p\oplus k[t]/p.
\end{equation*}
By Theorem~\ref{theorem:extension-matrix},
\begin{equation*}
  E_{(m+1,1,1)}(p)=\left\{\left. \bmat xyzW \right| x\in k[t]/p^{m+1}, \vec y,\vec z\in (k[t]/p)^2, W\in M_2(k[t]/p)\right\}.
\end{equation*}
and from the discussion at the beginning of Section~\ref{sec:matr-transf},
\begin{multline*}
  G_{(m+1,1,1)}(p)=\\\left\{\left. \gmat abcD\right| a\in (k[t]/p^{m+1})^*, \vec b,\vec c\in (k[t]/p)^2, D \in \GL_2(k[t]/p)\right\}.
\end{multline*}
\begin{theorem}
  A complete set of representatives for the $G_{(m+1,1,1)}(p)$-orbits in $E_{(m+1,1,1)}(p)$ is given by
  \begin{enumerate}
  \item \label{item:8} $\bmat x00W$, where $x\in k[t]/p^{m+1}$ with $p|x$, and $W\in M_2(k[t]/p)$ is in rational normal form; the isotropy group is
    \begin{equation*}
      \left\{\left.\gmat abcD\right|a\in (k[t]/p^{m+1})^*, \vec b,\vec c\in (k[t]/p)^2, D\in Z_{\GL_2(k[t]/p)}W\right\}.
    \end{equation*}
  \item \label{item:7} $\bmat x00W$, where $x\in k[t]/p^{m+1}$ with $(x,p)=1$, and $W\in M_2(k[t]/p)$ is in rational normal form;
    the isotropy group is
    \begin{equation*}
      \left\{\left.\gmat a00D\right|a\in (k[t]/p^{m+1})^*, D\in Z_{\GL_2(k[t]/p)}W\right\}.
    \end{equation*}
  \item \label{item:9} $\Mat x10ru000v$, where $x \in k[t]/p^m$ such that $p|x$, $r,u,v\in k[t]/p$, with $(r,p)=1$
    with isotropy group
    \begin{equation*}
      \left\{\left.\Mat ab0{bz}a000d\right| a\in (k[t]/p^{m+1})^*,b\in k[t]/p, d\in (k[t]/p)^*\right\}.
    \end{equation*}
  \item \label{item:11} $\Mat x1000u10v$, where $x \in k[t]/p^m$ such that $p|x$, $u,v\in k[t]/p$
    with isotropy group
    \begin{equation*}
      \left\{\left.\Mat ab{ud}{ud}a0{b+vd}da\right| a\in (k[t]/p^{m+1})^*,b,d\in k[t]/p\right\}.
    \end{equation*}
  \item \label{item:14} $\Mat x1000100w$, where $x \in k[t]/p^m$ such that $p|x$, $w\in k[t]/p$
    with isotropy group
    \begin{equation*}
      \left\{\left.\Mat a{b_1}{b_2}0a000a\right| a\in (k[t]/p^{m+1})^*,b_1,b_2\in k[t]/p\right\}.
    \end{equation*}
  \item \label{item:10} $\Mat x100u000w$, where $x \in k[t]/p^m$ such that $p|x$, $u,w\in k[t]/p$
    with isotropy group
    \begin{equation*}
      \left\{\left.\Mat a{b_1}{b_2}0a0{e(w-u)}de\right| a\in (k[t]/p^{m+1})^*,b_1,b_2,d\in k[t]/p, e\in (k[t]/p)^*\right\}.
    \end{equation*}
  \item \label{item:12} $\Mat x0010001w$, where $x \in k[t]/p^m$ such that $p|x$, $w\in k[t]/p$
    with isotropy group
    \begin{equation*}
      \left\{\left.\Mat a00{c_1}a0{c_2}0a\right| a\in (k[t]/p^{m+1})^*,c_1,c_2\in k[t]/p\right\}.
    \end{equation*}
  \item \label{item:13} $\Mat x001u000w$, where $x \in k[t]/p^m$ such that $p|x$, $u,w\in k[t]/p$
    with isotropy group
    \begin{equation*}
      \left\{\left.\Mat a0{d(u-w)}{c_1}ad{c_2}0e\right| a\in (k[t]/p^{m+1})^*,c_1,c_2,d\in k[t]/p, e\in (k[t]/p)^*\right\}.
    \end{equation*}
  \end{enumerate}
\end{theorem}
\begin{proof}
  In analogy with the proof of Theorem~\ref{theorem:m+1,1}, a combination of Birkhoff moves of type (\ref{eq:12}) and (\ref{eq:13}) respectively allows us to make transformations
  \begin{equation}
    \label{eq:transformations}
    \begin{array}{rcr}
      \bmat xyzW & \to &  \mat{x+p^m\transpose{\vec \alpha}z}{\transpose{\vec y}-x\transpose{\vec\alpha}}{\vec z}{W-\vec z\transpose{\vec\alpha}}\\
      \bmat xyzW & \to &  \mat{x-p^m\transpose{\vec y}\vec \alpha}{\vec y}{\vec z+x\vec\alpha}{W+\vec\alpha\transpose{\vec y}}
    \end{array}
  \end{equation}
  for any $\vec\alpha\in (k[t]/p)^2$.
  It follows that, if $(x,p)=1$, then Birkhoff moves of type (\ref{eq:12}) and (\ref{eq:13}) can be used to reduce to $\vec y=\vec z=0$.
  If $\vec y=\vec z=0$ (even if $(x,p)\neq 1$) transformations of the type $\gmat a00D$ can be used to reduce $D$ to its rational normal form. 
  Thus, when $(x,p)=1$ or $\vec y=\vec z=0$, $\bmat xyzW$ lies in an orbit of type \ref{item:8} or \ref{item:7}.
  The Birkhoff moves (\ref{eq:12})-(\ref{eq:15}) can not change the residue class of $x$ modulo $p^m$, therefore the classes of type \ref{item:7} are different from all the remaining types, where the residue of $x$ modulo $p$ is $0$.

  The matrices $\bmat x00W$ and $\bmat{x'}00{W'}$ are in the same orbit if and only if there exists $\gmat abcD\in G_{(m+1,1,1)}(p)$ such that
  \begin{equation*}
    \gmat abcD \bmat x00W = \bmat{x'}00{W'} \mat a{\transpose{\vec b}}{p^m\vec c}D,
  \end{equation*}
  This gives us $x\equiv x'\mod p^{m+1}$, and $DW\equiv W'D\mod p$.
  This shows that the matrices in \ref{item:8} and \ref{item:7} all represent distinct orbits.
  
  Two matrices $\bmat xyzW$ and $\bmat{x'}{y'}{z'}{W'}$, where $p|x$ and $p|x'$, are in the same $G_{(m+1,1,1)}(p)$-orbit if and only if there exists $\gmat abcD\in G_{(m+1,1,1)}(p)$ such that
  \begin{equation*}
    \gmat abcD \bmat xyzW =\bmat{x'}{y'}{z'}{W'} \mat a{\transpose{\vec b}}{p^m\vec c}D,
  \end{equation*}
  which simplifies to
  \begin{equation*}
    \mat{ax+p^m\transpose{\vec b}\vec z}{a\transpose{\vec y}}{D\vec z}{\vec c\transpose{\vec y}+DW}=\mat{ax'+p^m\transpose{\vec y'}\vec c}{\transpose{\vec y'}D}{a\vec z'}{\vec z'\transpose{\vec b}+W'D}.
  \end{equation*}
  Here multiples of $p$ are ignored in all but the top left entry.
  Thus the conditions for these matrices to lie in the same orbits are:
  \begin{gather}
    \label{eq:m+1}
    ax+ p^m\transpose{\vec b}\vec z\equiv ax' + p^m\transpose{\vec{y'}}\vec c\mod p^{m+1}\\
    \label{eq:17}
    \vec {y'}\equiv (a\inv \transpose D)\inv \vec y\mod p\\
    \label{eq:18}
    \vec {z'}\equiv a\inv D\vec z \mod p\\
    \label{eq:19}
    W'\equiv DWD\inv +(\vec c \;\transpose{\vec y}-\vec{z'}\transpose{\vec b})D\inv \mod p.
  \end{gather}
  The equation (\ref{eq:m+1}) implies that $x\equiv x'$ modulo $p^m$. 
  If either $\vec y\neq 0$ or $\vec z\neq 0$, and $x\equiv x'\mod p^m$, the transformations (\ref{eq:transformations}) can be used to make $x=x'$.
  Thus (\ref{eq:m+1}) can be replaced by the conditions
  \begin{gather}
    \label{eq:a}
    \tag{\ref{eq:m+1}a}
    x=x' \\
    \label{eq:16}
    \tag{\ref{eq:m+1}b}
    \transpose{\vec b}\vec z\equiv  \transpose{\vec{y'}}\vec c\mod p
  \end{gather}
  The equation (\ref{eq:17}) implies that if $\vec y$ is $0$ modulo $p$ for a matrix, then it is $0$ modulo $p$ for all matrices in its $G_{(2,1,1)}$-orbit.
  Similarly, the equation (\ref{eq:18}) implies that  if $\vec z$ is $0$ modulo $p$ for a matrix then it is zero modulo $p$ for all matrices in its $G_{(2,1,1)}(p)$-orbit.
  Thus classes of type \ref{item:8}, \ref{item:10} and~\ref{item:12} are all different, and differ from classes of type \ref{item:9} or \ref{item:11}.

  First consider the case where $\vec y\neq 0$.
  An appropriate choice of $a\inv D$ can be used to reduce to the case where $\vec y=(1,0)$.
  Now if $\vec y=\vec y'=(1,0)$, then $a\inv \transpose D$ has to be a matrix of the form $\mat 1{d_1}0{d_2}$.
  In other words, $a\inv D=\mat 10{d_1}{d_2}$.
  The orbits of $(k[t]/p)^2$ under the action of such matrices are represented by $(z_1,0)$ with $z_1\in k[t]/p$, $(0,1)$ and $(0,0)$, giving rise to matrices of type \ref{item:9}, \ref{item:10} and \ref{item:11}.

  Now consider each of these cases separately.

  If $\vec z=\vec{z'}=(z_1,0) \neq \vec 0$, then (\ref{eq:16}) gives $c_1=b_1z_1$.
  Substituting this in (\ref{eq:19}) gives
  \begin{equation*}
    W'=DWD\inv + \mat0{-z_1b_2}{c_2}0 D\inv.
  \end{equation*}  
  Also (\ref{eq:18}) gives that $D$ is of the form $\mat{d_1}00{d_2}$.
  Thus, the off-diagonal entries of $W$ can be made $0$, and different choices of the diagonal entries modulo $p$ give distinct orbits of type \ref{item:9}.

  If $\vec z=\vec{z'}=(0,1)$, then (\ref{eq:16}) gives $c_1=b_2$.
  Substituting this in (\ref{eq:19}) gives
  \begin{equation*}
    W'=DWD\inv + \mat{b_2}0{c_2-b_1}{-b_2} D\inv.
  \end{equation*}  
  With $D$ being the identity matrix, $W$ can be reduced to $\mat 0u0v$.
  By (\ref{eq:18}), $a\inv D=\mat 10{d_1}1$.
  Now if $W=\mat 0u0v$ and $W'=\mat0{u'}0{v'}$, then (\ref{eq:19}), multiplied on the right by $a\inv D$,  gives
  \begin{equation*}
    \mat {u'd_1}{u'}{v'd_1}{v'}=\mat{a\inv b_2}u{a\inv(c_2-b_1)}{ud_1+v-a\inv b_2}.
  \end{equation*}
  This forces $u=u'$ $v=v'$.
  It follows that the representatives of type \ref{item:11} represent all the distinct orbits for this case.

  If $\vec z=\vec{z'}=0$, then (\ref{eq:16}) gives $c_1=0$.
  This allows us to ignore the bottom-left entry and assume that $W$ is upper triangular, say $W=\mat{w_{11}}{w_{12}}0{w_{22}}$.
  For the description of $a\inv D$, it follows that $D$ is lower triangular.
  Thus $D$ can be written as the product of a diagonal matrix and a unipotent lower triangular matrix $\mat 10d1$.
  We have
  \begin{equation*}
    \mat 10d1 \mat{u}{v}0{w} \mat 10{-d}1=\mat{u-dv}{v}{du-d^2v-dw}{dv+w}.
  \end{equation*}
  Therefore, if $v\neq 0$, then such a transformation can be used to make $u=0$.
  Taking $D$ diagonal can then be used to make $v=1$, and the bottom right left can be reset to $0$ as before.
  If, on the other hand, $v=0$, then no further reduction is possible.
  This gives rise to the cases \ref{item:14} and \ref{item:10}.

  Now consider the final case where $\vec y=0$, and $\vec z\neq 0$.
  By a suitable choice of $a$ and $D$, (\ref{eq:18}) can be used to get $\vec z=(1,0)$.
  Now the analysis is analogous to the case where $\vec y=(1,0)$ and $\vec z=0$:
  (\ref{eq:16}) gives $b_1=0$, which, using (\ref{eq:19}) with $D$ as the identity allows us to assume that $W$ is lower triangular, say $W=\mat u0vw$.
  If $\vec z=\vec z'=0$, then (\ref{eq:18}) implies that $D$ is upper triangular.
  Unipotent upper triangular $D$ can be used to make $u\equiv 0 \mod p$ provided that $(v,p)=1$.
  Diagonal $D$ can be used to make $v\equiv 1\mod p$.
  On the other hand, if $p|v$, no further reductions are possible.
  This gives the cases \ref{item:12} and \ref{item:13}.

  In order to obtain the descriptions of the remaining centralizers, it suffices to set $\vec y=\vec y$, $\vec z=\vec z'$, $W=W'$, and substitute each with the normal forms of \ref{item:9}-\ref{item:13}.
  Since these calculations are similar to, and simpler than the ones above, they are omitted.
\end{proof}
The number of classes with various cardinalities of centralizers when $k[t]/p$ is finite of order $q$ are displayed in Table~\ref{tab:m11}, along with the number of self-transpose classes.
\begin{table}
  \centering
  \caption{$\lambda = (m+1,1,1)$}
  \label{tab:m11}
  \footnotesize
  \begin{equation*}
    \begin{array}{|c|c|c|c|}
      \hline
      \text{Reference} &\text{Centralizer} & \text{Total Classes} & \text{Self-transpose}\\
      \hline
      \text{\ref{item:8}}& q^{m+9}(1-q\inv)^2(1-q^{-2}) & q^{m+1} & \text{ all }\\
      \hline
      \text{\ref{item:8}}&q^{m+7}(1-q\inv)^3 & q^{m+1}(q-1)/2 & \text{ all } \\
      \hline
      \text{\ref{item:8}}&q^{m+7}(1-q\inv)^2 & q^{m+1} & \text{ all }\\
      \hline
      \text{\ref{item:8}}&q^{m+7}(1-q\inv)(1-q^{-2})& q^{m+1}(q-1)/2 & \text{ all }\\
      \hline
      \text{\ref{item:7}}&q^{m+5}(1-q\inv)^2(1-q^{-2})& q^{m+2}(1-q\inv) & \text{ all }\\
      \hline
      \text{\ref{item:7}}&q^{m+3}(1-q\inv)^3& q^{m+1}(q-1)^2/2 & \text{ all }\\
      \hline
      \begin{matrix}
        \text{\ref{item:7}}\\
        \text{\ref{item:9}}
      \end{matrix}
      &q^{m+3}(1-q\inv)^2& 2q^{m+1}(q-1) & \text{all}\\
      \hline
      \text{\ref{item:7}}&q^{m+3}(1-q\inv)(1-q^{-2}) & q^m(q-1)(q^2-q)/2&\text{all}\\
      \hline
      \begin{matrix}
        \text{\ref{item:11}}\\
        \text{\ref{item:14}}\\
        \text{\ref{item:12}}
      \end{matrix}
      & q^{m+3}(1-q\inv) & q^{m+1}+2q^m & q^{m+1}\\
      \hline
      \begin{matrix}
        \text{\ref{item:10}}\\
        \text{\ref{item:13}}
      \end{matrix}
      &q^{m+5}(1-q\inv)^2 & 2q^{m+1} & 0\\
      \hline
      \hline
      & \text{Totals:} &  q^m(q^3 + 2q^2 + 2q + 2) & q^m(q^3+2q^2)\\
      \hline
    \end{array}
  \end{equation*}
\end{table}

\section{Class Equation}
\label{sec:class-equation}
Now let $R$ be a finite local principal ideal ring of length two with residue field of cardinality $q$.
We are now ready to enumerate the similarity classes in $M_n(R)$ along with the cardinalities of their centralizers for $n=2,3,4$.

Recall from Section~\ref{sec:analysis-of-fibres} that each similarity class in $M_n(k)$  can be identified with a function $c:\Irr kt\to \Lambda$  satisfying (\ref{eq:20}).
\begin{defn}[Type]
  Two similarity classes $c$ and $c'$ in $M_n(k)$ are said to be of the same type if there exists a degree-preserving permutation $\phi:\Irr kt\to \Irr kt$ such that $c'=c\circ \phi$.
\end{defn}
Thus the type of a class only remembers the degree of each irreducible polynomial $p$ and the combinatorial data $\phi(p)$ coming from the rational normal form of the $p$-primary part of the class.
\begin{defn}[Primary Type]
  A primary type of size $n$ is a pair $(\lambda,d)$, where $\lambda$ is a partition and $d$ is a positive integer (the degree) such that $|\lambda|d=n$.
  For brevity of notation the primary type $(\lambda,d)$ will be denoted by $\lambda_d$.
\end{defn}
For each $c\in C(n)$, let $\tau(c)$ denote the multiset of primary types obtained by listing out the symbols $\phi(f)_{\deg f}$ as $f$ ranges over the set $\Irr kt$.
Clearly,
\begin{lemma}
  Two classes $c$ and $c'$ in $C(n)$ are of the same type if $\tau(c)=\tau(c')$.
\end{lemma}
\begin{example}
  There are four types of classes of $2\times 2$ matrices:
  \begin{enumerate}
  \item $(1,1)_1$ (central type)
  \item $(2)_1$ (non-semisimple type)
  \item $(1)_1,(1)_1$ (split regular semisimple type)
  \item $(1)_2$ (irreducible type)
  \end{enumerate}
\end{example}
Some combinatorial invariants of a similarity class can be computed in terms of its type as follows:
for a partition $\lambda=(\lambda_1,\lambda_2,\dotsc,\lambda_r)$ (parts written in decreasing order), define
\begin{equation*}
  m(\lambda) = \lambda_r + 3\lambda_{r-1} + \dotsb + (2r-1)\lambda_1.
\end{equation*}
One may deduce the following well-known lemma from the description of centralizers of matrices in Singla's thesis \cite[Section~2.3]{Poojathesis}.
\begin{theorem}
  \label{theorem:centralizer-cardinality}
  The centralizer algebra of a matrix of primary type $(\lambda,d)$ is a vector space of dimension $dm(\lambda)$.
  If the entries are in the finite field $\Fq$, then the centralizer group in $\GL_n(\Fq)$ is of order
  \begin{equation*}
    q^{dm(\lambda)}\prod_i(1-q^{-d})(q-q^{-2d})\dotsb(1-q^{-(m_i-1)d})
  \end{equation*}
  where, for each positive integer $i$, $m_i$ denotes the number of times that $i$ occurs in $\lambda$.
\end{theorem}
For an arbitrary type, the centralizer algebra (or centralizer group) is the product of the centralizer algebras (or groups) of the primary types that occur in it.
In particular, the dimension of the centralizer algebra is given by
\begin{equation}
  \label{eq:ztau}
  z_\tau = \sum_{(\lambda,d)} dm(\lambda),
\end{equation}
the sum being over all the primary types in $\tau$.
\begin{example}
  For the type $\tau = (2,1,1)_1, (1)_2$ (which is of size 6), the centralizer algebra is of dimension
  \begin{equation*}
    z_\tau = m(2,1,1) + 2m(1) = 12.
  \end{equation*}
  For $k=\Fq$, the centralizer group has cardinality
  \begin{equation*}
    [q^{10}(1-q\inv)^2(1-q^{-2})][q^2(1-q^{-2})].
  \end{equation*}
\end{example}
\renewcommand{\arraystretch}{1.2}
\begin{table}
  \centering
  \caption{Similarity classes in $M_2(R)$}
  \label{tab:2x2}
  \begin{equation*}
    \begin{array}{|c|c|}
      \hline
      \text{Cardinality} & \text{Number of classes}
      \\
      \hline
      q^{4} + q^{3} & \frac{1}{2} q^{4} - \frac{1}{2} q^{3} \\ 
      \hline
      q^{4} - q^{2} & q^{3} \\ 
      \hline
      q^{2} + q & \frac{1}{2} q^{3} - \frac{1}{2} q^{2} \\ 
      \hline
      q^{2} - 1 & q^{2} \\ 
      \hline
      1 & q^{2} \\ 
      \hline
      q^{2} - q & \frac{1}{2} q^{3} - \frac{1}{2} q^{2} \\ 
      \hline
      q^{4} - q^{3} & \frac{1}{2} q^{4} - \frac{1}{2} q^{3} \\ 
      \hline
    \end{array}
  \end{equation*}
\end{table}
\begin{table}
  \centering
  \caption{Similarity classes in $M_3(R)$}
  \label{tab:3x3}
  \begin{equation*}
    \begin{array}{{|c|c|c|}}
      \hline
      \text{Cardinality} & \text{Number of classes} & \text{Selftranspose classes}\\
      \hline
      1 & q^{2} & q^{2} \\
      \hline
      q^{8} + q^{7} - q^{5} - q^{4} & q^{3} & q^{3} \\
      \hline
      q^{6} + 2 q^{5} + 2 q^{4} + q^{3} & \frac{1}{6} q^{4} - \frac{1}{2} q^{3} + \frac{1}{3} q^{2} & \frac{1}{6} q^{4} - \frac{1}{2} q^{3} + \frac{1}{3} q^{2} \\
      \hline
      q^{10} + 2 q^{9} + 2 q^{8} + q^{7} & \frac{1}{2} q^{5} - q^{4} + \frac{1}{2} q^{3} & \frac{1}{2} q^{5} - q^{4} + \frac{1}{2} q^{3} \\
      \hline
      q^{12} - q^{10} - q^{9} + q^{7} & q^{4} & q^{4} \\
      \hline
      q^{10} - q^{7} & \frac{1}{2} q^{5} - q^{4} + \frac{1}{2} q^{3} & \frac{1}{2} q^{5} - q^{4} + \frac{1}{2} q^{3} \\
      \hline
      q^{12} - q^{9} & \frac{1}{2} q^{6} - \frac{1}{2} q^{5} & \frac{1}{2} q^{6} - \frac{1}{2} q^{5} \\
      \hline
      q^{6} - q^{3} & \frac{1}{2} q^{4} - \frac{1}{2} q^{3} & \frac{1}{2} q^{4} - \frac{1}{2} q^{3} \\
      \hline
      q^{12} - q^{11} - q^{10} + q^{9} & \frac{1}{3} q^{6} - \frac{1}{3} q^{4} & \frac{1}{3} q^{6} - \frac{1}{3} q^{4} \\
      \hline
      q^{6} - q^{4} - q^{3} + q & q^{2} & q^{2} \\
      \hline
      q^{4} + q^{3} + q^{2} & q^{3} - q^{2} & q^{3} - q^{2} \\
      \hline
      q^{6} + q^{5} - q^{3} - q^{2} & q^{3} - q^{2} & q^{3} - q^{2} \\
      \hline
      q^{4} + q^{3} - q - 1 & q^{2} & q^{2} \\
      \hline
      q^{6} - q^{5} - q^{4} + q^{3} & \frac{1}{3} q^{4} - \frac{1}{3} q^{2} & \frac{1}{3} q^{4} - \frac{1}{3} q^{2} \\
      \hline
      q^{8} + q^{7} + q^{6} & q^{4} - q^{3} & q^{4} - q^{3} \\
      \hline
      q^{10} + q^{9} - q^{7} - q^{6} & 2 q^{4} - 2 q^{3} & 2 q^{4} - 2 q^{3} \\
      \hline
      q^{12} + q^{11} - q^{9} - q^{8} & q^{5} - q^{4} & q^{5} - q^{4} \\
      \hline
      q^{12} + 2 q^{11} + 2 q^{10} + q^{9} & \frac{1}{6} q^{6} - \frac{1}{2} q^{5} + \frac{1}{3} q^{4} & \frac{1}{6} q^{6} - \frac{1}{2} q^{5} + \frac{1}{3} q^{4} \\
      \hline
      q^{10} - q^{8} - q^{7} + q^{5} & q^{3} + q^{2} & q^{3} - q^{2} \\
      \hline
    \end{array}
  \end{equation*}
\end{table}
\begin{theorem}
  \label{theorem:counting-classes}
  Let $k$ be a finite field of order $q$.
  For every positive integer $N$, the number of similarity classes in $M_n(R)$ whose centralizer in $\GL_n(R)$ has $N$ elements is
  \begin{equation*}
    \sum_\tau n_\tau\#\big\{\xi\in \big(G_{A_\tau}\bsl E_{k[t]}(M^{A_\tau},M^{A_\tau})\big)\::\:q^{z_\tau}|G_\xi|=N\big\},
  \end{equation*}
  where the sum is over all types $\tau$ of similarity classes in $M_n(k)$, and for each $\tau$, $n_\tau$ denotes the number of similarity classes in $M_n(k)$ of type $\tau$, $z_\tau$ is given by (\ref{eq:ztau}), and $A_\tau$ represents any matrix of type $\tau$.
  Similarly, the number of conjugacy classes in the group $\GL_n(R)$  whose centralizer has $N$ elements is
  \begin{equation*}
    \sum_\tau n^*_\tau\#\big\{\xi\in \big(G_{A_\tau}\bsl E_{k[t]}(M^{A_\tau},M^{A_\tau})\big)\::\:q^{z_\tau}|G_\xi|=N\big\},
  \end{equation*}
  where $n^*_\tau$ denotes the number of invertible similarity classes in $M_n(k)$ of type $\tau$.
\end{theorem}
\begin{remark}
  The numbers $n_\tau$ and $n^*_\tau$ are not difficult to compute:  
  the number $\Phi_n(q)$ of irreducible monic polynomials of degree $n$ in $\Fq[t]$ is given by the following well-known formula:
  \begin{equation*}
    \Phi_n(q) = \frac 1n\sum_{d|n} \mu(n/d)q^d,
  \end{equation*}
  where $\mu$ denotes the classical M\"obius function.
  For each degree $d$, let $n_d(\tau)$ denote the number of primary types in $\tau$ of degree $d$.
  If $d=1$, let $\Phi^*_d(q)=\Phi_d(q)-1$, otherwise let $\Phi^*_d(q)=\Phi_d(q)$ (thus $\Phi^*_d(q)$ is the number of irreducible polynomials of degree $d$ other than the linear polynomial $t$).
  Then
  \begin{equation*}
    n_\tau = \prod_d \binom{\Phi_d(q)}{n_d(\tau)}; \quad n^*_\tau = \prod_d \binom{\Phi^*_d(q)}{n_d(\tau)}.
  \end{equation*}
\end{remark}

Thus, Theorem~\ref{theorem:counting-classes}, along with Theorem~\ref{theorem:types} and the results of Section~\ref{sec:description-orbits}, allows us to compute the cardinalities of all the similarity classes in $M_n(R)$ for $n\leq 4$.
Tables~\ref{tab:2x2} and~\ref{tab:3x3} give the decomposition of $M_n(R)$ into similarity classes for $n=2$ and $3$ (for $n=2$, all classes are self-transpose).
These tables were generated by computer using a Sage program.
A similar table can be generated for $n=4$, and also for $n\leq 4$ counting only invertible similarity classes.
One may also compute the number of similarity classes in $M_n(R)$ which lie above the class of a given matrix $A\in M_n(k)$ in terms of the type of $A$, provided that the results of Section~\ref{sec:description-orbits} cover all the partitions which occur in the type of $A$.
Among these the self-transpose classes can also be enumerated.

If we were only interested in computing the total number of similarity classes, we would need to know, 
for each similarity class type $\tau$, the quantity
\begin{equation*}
  c_\tau(q) = |G_{A_\tau}\bsl E_{k[t]}(M^{A_\tau}, M^{A_\tau})|,
\end{equation*}
where $A_\tau$ is any matrix of type $\tau$.
From this we may compute the total number of similarity classes in $M_n(R)$ (which can also be thought of as $c_{(2^n)}(q)$) using the identity:
\begin{equation}
  \label{eq:23}
  c_{(2^n)}(q) = \sum_\tau n_\tau(q)c_\tau(q).
\end{equation}
We have already seen that $c_\tau$ does not depend on the choice of $A_\tau$.
Explicitly, if $\tau = (\lambda^{(1)}_{d_1}, \lambda^{(2)}_{d_2}, \dotsc)$, then
\begin{equation}
  \label{eq:21}
  c_\tau(q) = \prod_i c_{\lambda^(i)}(q^{d_i}),
\end{equation}
where, for any partition $\lambda$,
\begin{equation}
  \label{eq:22}
  c_\lambda(q) = |G_{N_\lambda}\bsl E_{k[t]}(M^{N_\lambda}, M^{N_\lambda})|,
\end{equation}
where, for any partition $\lambda$, $N_\lambda$ denotes the nilpotent matrix in Jordan canonical form with block sizes equal to the parts of $\lambda$.

Using the fact that the $c_\tau(q)$'s can be computed using equations (\ref{eq:21}) and (\ref{eq:22}), the identity (\ref{eq:23}) can be rewritten as a beautiful product formula for the generating function using the cycle index techniqes of Kung \cite{kung1981cycle} and Stong \cite{stong1988some}:
\begin{equation*}
  \sum_{n=0}^\infty c_{(2^n)}(q) x^n = \prod_{d = 1}^\infty \Big(\sum_{\lambda \in \Lambda} c_\lambda(q^d)x^d\Big)^{\Phi_d(q)}.
\end{equation*}
Recall that the symbol $\Lambda$ in the above equation denotes the set of all partitions, including the partition $\emptyset$ of $0$.
Since $c_\emptyset(q) = 1$, each factor in the above product has constant term $1$.
If $c_\lambda(q)$ is replaced by the number of self-transpose orbits in $E_{k[t]}(M^{N_\lambda}, M^{N_\lambda})$, then the generating function for self-transpose similarity classes in $M_n(R)$ is obtained.
\begin{table}[h]
  \caption{Number of similarity classes}
  \label{tab:classes}
  \begin{tabular}{|c|c|}
    \hline
    classes in & number\\
    \hline
    $M_2(R)$ & $q^4+q^3+q^2$\\
    $\GL_2(R)$ & $q^4-q$\\
    $M_3(R)$ & $q^6+q^5+2q^4+q^3+2q^2$\\
    $\GL_3(R)$ & $q^6-q^3+2q^2-2q$\\
    $M_4(R)$ & $q^8 + q^7 + 3q^6 + 3q^5 + 5q^4 + 3q^3 + 3q^2$\\
    $\GL_4(R)$ & $q^8 + q^6 - q^5 + 2q^4 - 2q^3 + 2q^2 - 3q$\\
    \hline
  \end{tabular}
\end{table}
\begin{table}[h]
  \caption{Number of self-transpose similarity classes}
  \label{tab:self-transpose-classes}
  \begin{tabular}{|c|c|}
    \hline
    classes in & number\\
    \hline
    $M_2(R)$ & $q^4+q^3+q^2$\\
    $\GL_2(R)$ & $q^4-q$\\
    $M_3(R)$ & $q^{6} + q^{5} + 2 q^{4} + q^{3}$\\
    $\GL_3(R)$ & $q^{6} - q^{3}$\\
    $M_4(R)$ & $q^8 + q^7 + 3q^6 + 3q^5 + 3q^4 + q^3 + q^2$\\
    $\GL_4(R)$ & $q^8 + q^6 - q^5 - q$\\
    \hline
  \end{tabular}
\end{table}
Tables~\ref{tab:classes} and~\ref{tab:self-transpose-classes} give the total number of similarity classes and self-transpose similarity classes in $M_n(R)$ and $GL_n(R)$ for $n\leq 4$.
\section{Case where $R$ is not principal}
\label{generalring}
Let $R$ be a local ring with maximal ideal $P$, not necessarily principal, satisfying $P^2=0$.  Write $k$ again for the residue field.  In this section we sketch how the results of this paper may be adapted to understand the conjugacy classes of $M_n(R)$.     Note that $P$ may be viewed as a $k$-vector space, and that we have the exact sequence
\begin{equation} \label{bank}
  \ses{P}{R}{k}{\iota}{p}
\end{equation}
of $R$-modules.  Here $\iota$ is the usual inclusion.
Let $\A \in M_n(R)$ with $r(\A)=A \in M_n(k)$, and write $M^{\A}$ as before for the $R[t]$-module structure on $R^n$ determined by $\A$.  Tensoring the above sequence with $M^{\A}$ gives
\begin{equation}
  \ses {M^A \otimes_k P}{M^{\A}}{M^A}{\iota}p.
\end{equation}
As $R$-modules this is identical with the $n$-fold direct sum
\begin{equation}
  \tag{$\xi_0$}
  \ses {P^n}{R^n}{k^n}{\iota}p
\end{equation}
of (\ref{bank}).

Using the methods of Section \ref{sec:orbit-extens-assoc} we obtain
\begin{equation*}
    0\to \Ext_{k[t]}(M^A,M^A) \otimes_k P  \stackrel\iota\to  \Ext_{R[t]}(M^A \otimes_k P,M^A) \stackrel\omega\to \Ext_R(M^A \otimes_k P,M^A).
\end{equation*}

As before, one finds that $G_A$ acts (diagonally) on all terms of this sequence, and that $\omega \inv(\xi_0)$ is a sub-$G_A$-set of $\Ext_{R[t]}(M^A \otimes_k P,M^A)$.
Lemma \ref{lemma:fiber} and Theorem \ref{theorem:fiber} hold as stated, reducing the problem to finding the $G_A$-orbits in $\omega\inv(\xi_0)$.

As before we obtain:
\begin{theorem}
  There exists an bijection between the sets $\omega\inv(\xi_0)$ and $E_{k[t]}(M^A,M^A) \otimes_k P$ which preserves the action of $G_A$.
\end{theorem}

From this we obtain:
\begin{cor} 
  \label{cor:nonprincipal}
  The similarity classes in $M_n(R)$ depend on $R$ only through $k$ and the dimension of $P$ over $k$.
\end{cor}

To understand $E_{k[t]}(M^A,M^A) \otimes_k P$ we may reduce to the case in which $A=C_{\lambda}(p)$ as before.

Suppose that $P$ is finite-dimensional as a $k$-space, of dimension $d$.   The $k$-space $E_{k[t]}(M^A,M^A) \otimes_k P$ can be understood as a $d$-fold copy of the spaces $E_{\lambda}(p)$, considered with the {\it diagonal} action of $G_{\lambda}(p)$.  Thus, the problem is to determine the equivalence classes of $d$-tuples of relations matrices under this diagonal action.

To give an example, let $k$ be a field, and $d \geq 1$.  Put
\begin{equation}
  R=k[x_1, \ldots, x_d]/(x_1,\ldots, x_d)^2.
\end{equation}
Then $R$ is a local ring with maximal ideal $P=(x_1,\ldots, x_d)$ satisfying $P^2=0$, and residue field $k$.  Moreover $P$ has dimension $d$ as a $k$-space.  Let $A=0 \in M_n(k)$.  Then the $\GL_n(R)$-conjugacy classes of fibres over $A$ correspond to equivalence classes of $d$-tuples of matrices in $M_n(k)$ under conjugation by $\GL_n(k)$.  In the case that $d=2$ this is exactly the matrix pair problem, which as noted by Drozd \cite{MR607157} is a ``classical unsolved problem'' of ``extreme difficulty''.  We therefore do not pursue it further here.

\subsection*{Acknowledgements}
The first author thanks Joseph Oesterl\'e, Uday Bhaskar Sharma, and S. Viswanath for some helpful conversations.  The third author thanks Neil Epstein for suggesting Section \ref{generalring}.  He is also grateful to Patrick Soboleski, whose undergraduate research project motivated Section \ref{sec:conj-transp}.
All the authors thank an anonymous refereee who made some valuable suggestions and pointed out numerous corrections.

\def\cprime{$'$}

\end{document}